\documentclass[]{interact}

\usepackage{epstopdf}
\usepackage[caption=false]{subfig}

\usepackage[numbers,sort&compress]{natbib}
\bibpunct[, ]{[}{]}{,}{n}{,}{,}

\usepackage{tikz}
\graphicspath{{figure/}{analysis/}}

\theoremstyle{plain}
\newtheorem{theorem}{Theorem}[section]
\newtheorem{lemma}[theorem]{Lemma}

\newtheorem{problem}[theorem]{Problem}

\theoremstyle{definition}
\newtheorem{definition}[theorem]{Definition}

\theoremstyle{remark}
\newtheorem{remark}{Remark}

\newcommand{\norm}[1]{\left\Vert #1 \right\Vert}

\newcommand{\Prb}[1]{\mathbb{P}\left[ #1 \right]}

\newcommand{\E}[1]{\mathbb{E}\left[ #1 \right]}

\newcommand{\cond}[2]{\mathbb{E}\left[\left. #1 \right\vert #2 \right]}

\newcommand{\argmin}{\mathrm{argmin}}
\newcommand{\mrank}{\texttt{rank}}
\newcommand{\mnull}{\texttt{null}}
\newcommand{\mrow}{\texttt{row}}
\newcommand{\mcol}{\texttt{col}}

\begin{document}

\articletype{}

\title{The Impact of Local Geometry and Batch Size on Stochastic Gradient Descent for Nonconvex Problems}

\author{
\name{Vivak Patel\textsuperscript{a}\thanks{Email: vp314@uchicago.edu}}
\affil{\textsuperscript{a}Department of Statistics, University of Chicago, Illinois, USA}
}

\maketitle

\begin{abstract}
In several experimental reports on nonconvex optimization problems in machine learning, stochastic gradient descent (SGD) was observed to prefer minimizers with flat basins in comparison to more deterministic methods, yet there is very little rigorous understanding of this phenomenon. In fact, the lack of such work has led to an unverified, but widely-accepted stochastic mechanism describing why SGD prefers flatter minimizers to sharper minimizers. However, as we demonstrate, the stochastic mechanism fails to explain this phenomenon. Here, we propose an alternative deterministic mechanism that can accurately explain why SGD prefers flatter minimizers to sharper minimizers. We derive this mechanism based on a detailed analysis of a generic stochastic quadratic problem, which generalizes known results for classical gradient descent. Finally, we verify the predictions of our deterministic mechanism on two nonconvex problems.
\end{abstract}


\begin{amscode}
90C15; 90C30
\end{amscode}

\section{Introduction}
Stochastic gradient descent (SGD) is a stochastic optimizer that has had tremendous success for nonconvex problems in machine learning \cite{bottou2016}. SGD's success has resulted in a number of experimental inquiries into why SGD outperforms its more deterministic counterparts, one of which noted that stochastic methods tend to converge to flatter minimizers---minimizers where the Hessian has relatively small eigenvalues---in comparison to deterministic methods \citep{keskar2016}. While the flatness of a minimizer's impact on machine learning performance metrics is still an open question \cite{dinh2017}, understanding SGD's observed preference for flatter minimizers can shed light into how this method operates for nonconvex problems, and (a) how to potentially leverage this understanding to induce SGD iterates to converge to minimizers with particular properties \cite{hochreiter1995,chaudhari2016}, or (b) how to formulate verifiable stopping criteria \cite{patel2016}.

Unfortunately, a rigorous analysis of SGD on nonconvex problems is still missing. While there are notable attempts that have analyzed structurally distinct variants of SGD on nonconvex problems to understand how these variants escape saddle points \cite{ge2015,jin2017,mou2017,zhang2017}, and even how these variants diverge from local minimizers \cite{kleinberg2018},\footnote{\cite{kleinberg2018} was posted to the arXiv several months after an earlier version of this work was posted to the arXiv.} these analyses do not actually apply to SGD. More relevant analyses are the works of \cite{hardt2015}, which established the algorithmic stability of SGD for nonconvex problems, and \cite{bottou2016}, which established the convergence of SGD to stationary points of nonconvex problems in probability. However, these results do not provide insight into why SGD prefers flatter minimizers. 

Absent a rigorous analysis, the widely-accepted explanation for why SGD prefers flat minimizers in comparison to sharp minimizers is, what we will call, the stochastic mechanism \cite{keskar2016,bala2017}. Roughly, the stochastic mechanism states that the stochasticity of SGD forces its iterates to explore the objective function's landscape, which, in turn, reduces the probability with which it can be attracted to sharp minimizers. The stochastic mechanism is better understood through a simple example, which is adapted from \cite{keskar2016}. 

\subsection{The Stochastic Mechanism}

Consider using SGD on a one-dimensional stochastic nonconvex problem whose expected objective function is given by the graph in Figure \ref{figure:one-dim-obj-A}. The stochasticity of the nonconvex problem is generated by adding independent symmetric, bounded random variables to the derivative evaluation at any point. That is, at any iteration, SGD's search direction is generated by the derivative of the expected objective function plus some symmetric, bounded random noise. 

Intuitively, the noise in the SGD search directions, by the size of the basins alone, will force the iterates to spend more time in the flatter minimizer on the left. That is, the noise of in SGD's search directions will increase an iterates probability of being in the flatter basin in comparison to the sharper basin. This is precisely stochastic mechanism.

\begin{figure}[bh]
\centering
\begin{tikzpicture}
\draw[step=1cm,gray,very thin,dotted] (-6,-1) grid (3,4);
\draw (-3,1) parabola (-6,4);
\draw (-3,1) parabola (1,4);
\draw (2,-1) parabola (1,4);
\draw (2,-1) parabola (3,4);
\end{tikzpicture}
\caption{The expected objective function for a simple one-dimensional nonconvex problem. The minimizer on the left is the flatter minimizer. The minimizer on the right is the sharper minimizer.}
\label{figure:one-dim-obj-A}
\end{figure}
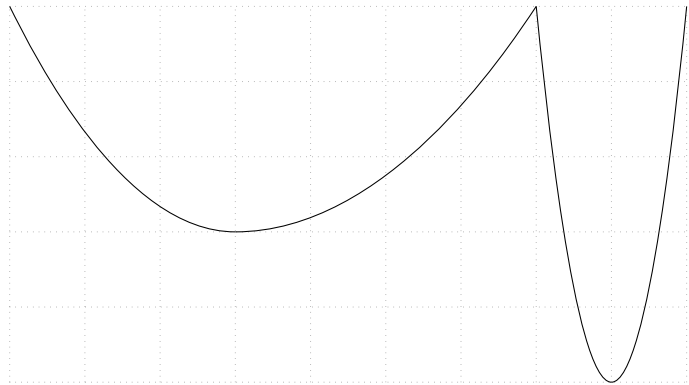

For another example, consider a similar one-dimensional stochastic nonconvex problem whose expected objective function is given by the graph in Figure \ref{figure:one-dim-obj-B}. In this example, the minimizer on the left is the flatter minimizer while the minimizer on the right is the sharper minimizer. Now the stochastic mechanism would predict that the probability of an iterate being in the basin of the sharper minimizer is actually greater than that of the flatter minimizer. 

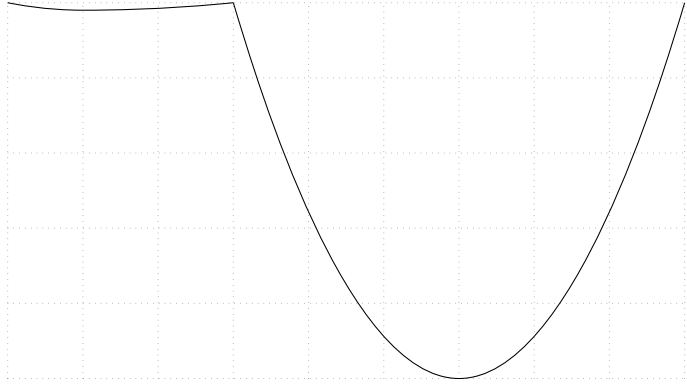
\begin{figure}
\centering
\begin{tikzpicture}
\draw[step=1cm,gray,very thin,dotted] (-6,-1) grid (3,4);
\draw (-5,3.9) parabola (-6,4);
\draw (-5,3.9) parabola (-3,4);
\draw (0,-1) parabola (-3,4);
\draw (0,-1) parabola (3,4);
\end{tikzpicture}
\caption{The expected objective function for a different one-dimensional nonconvex problem. The minimizer on the left is the flatter minimizer. The minimizer on the right is the sharper minimizer.}
\label{figure:one-dim-obj-B}
\end{figure} 

Thus, the stochastic mechanism is not explaining why SGD prefers flatter minimizers, but rather it is explaining why SGD has a higher probability of being in basins with larger size. Moreover, while flatness and basin size might seem to be intuitively connected, as Figure \ref{figure:one-dim-obj-B} shows, this is not necessarily the case. Therefore, according to this example, we might assume that flatness is an inappropriate term, and it is actually the volume of the minimizer that is important (based on the stochastic mechanism) in determining to which minimizers the SGD iterates converge.

However, ignoring flatness or sharpness of a minimizer in favor of the size of its basin of attraction does not align with another experimental observation: in high-dimensional problems, SGD even avoids those sharper minimizers that had a small fraction of their eigenvalues termed relatively ``large'' \cite{keskar2016}. Importantly, in such high-dimensions, the probability of the noise exploring those limited directions that correspond to the ``larger'' eigenvalues is low (e.g., consider isotropic noise in large dimensions). This suggests that, while the stochastic mechanism is intuitive and can explain why SGD may prefer large basins of attractions, it is insufficient to explain why SGD prefers flatter minimizers over sharper minimizers. Accordingly, what is an appropriate mechanism to describe how SGD iterates ``select'' minimizers?

\subsection{Main Contribution}
In this work, we derive and propose a deterministic mechanism based on the expected local geometry of a minimizer and the batch size of SGD to describe why SGD iterates will converge or diverge from the minimizer. As a consequence of our deterministic mechanism, we can explain why SGD prefers flatter minimizers to sharper minimizers in comparison to gradient descent, and we can explain the manner in which SGD will ``escape'' from certain minimizers. Importantly, we verify the predictions of our deterministic mechanism with numerical experiments on two nontrivial nonconvex problems. 

\subsection{Outline}

In \S \ref{section:classical}, we begin by reviewing the classical mechanism by which gradient descent diverges from a quadratic basin. In \S \ref{section:stochastic-quadratic}, we will rigorously generalize the classical, deterministic mechanism to SGD-$k$ (i.e., SGD with a batch size of $k$) on a generic stochastic quadratic problem, and, using this rigorous analysis, we define our deterministic mechanism. In \S \ref{section:homogeneous} and \S \ref{section:inhomogeneous}, we numerically verify the predictions of our deterministic mechanism on two nontrivial, nonconvex problem. In \S \ref{section:conclusion}, we conclude this work. 

\section{Classical Gradient Descent} \label{section:classical}
As is well known, classical gradient descent also has a mechanism to diverge from local minimizers. To illustrate this mechanism, consider a simple one-dimensional quadratic problem $f(x) = x^2$ that is minimized iteratively using classical gradient descent with a fixed step size of $C = 1.1$. Suppose that the procedure is initialized at $x_0 = -1$. Then, the next iterate is given by $x_1 = x_0 - C(2x_0) = 1.2$. The second iterate is given by $x_2 = x_1 - C(2x_1) = 1.2 - (1.1)(2.4) = -1.44$. This pattern is illustrated in Figure \ref{figure:divergence-gd}.

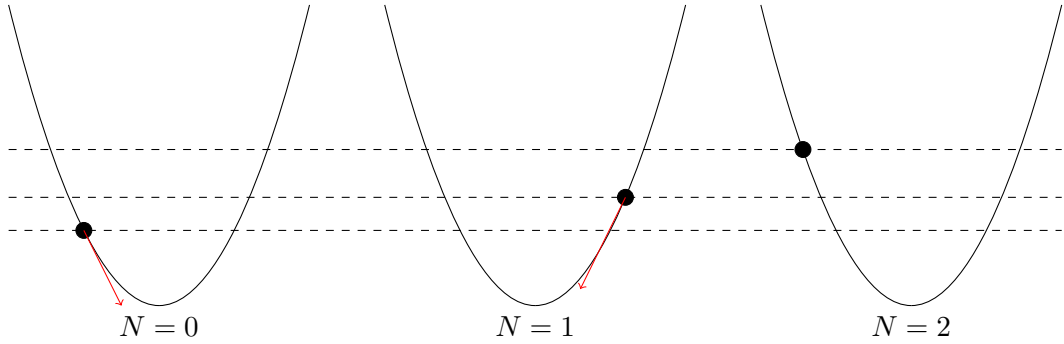
\begin{figure}[hbt]
\centering
\begin{tikzpicture}
\draw (-5,0) parabola (-7,4);
\draw (-5,0) parabola (-3,4);
\node[anchor=north] at (-5,0) {$N=0$};
\draw[dashed] (-7,1) -- (7,1);
\filldraw[black] (-6,1) circle (3pt) node[anchor=north east] {};
\draw[red,->] (-6,1) -- (-5.5,0);

\draw (0,0) parabola (-2,4);
\draw (0,0) parabola (2,4);
\node[anchor=north] at (0,0) {$N=1$};
\draw[dashed] (-7,1.44) -- (07,1.44);
\filldraw[black] (1.2,1.44) circle (3pt) node[anchor=south east] {};
\draw[red,->] (1.2,1.44) -- (0.6,0.22);

\draw (5,0) parabola (3,4);
\draw (5,0) parabola (7,4);
\node[anchor=north] at (5,0) {$N=2$};
\draw[dashed] (-7,2.0763) -- (7,2.0763);
\filldraw[black] (5-1.44,2.0763) circle (3pt) node[anchor=north west] {};
\end{tikzpicture}
\caption{An illustration of the divergence of gradient descent applied to $f(x)=x^2$, initialized at $x_0 = -1$ with a step size of $C = 1.1$. The arrows show tangent ray in the negative derivative direction.}
\label{figure:divergence-gd}
\end{figure}

There are two important observations that emerge from this simple example. The first observation is that, at each iteration, gradient descent is moving further and further away from the minimizer at $x=0$. This is a consequence of our choice in step size. In fact, if we had selected a step size of $C = 0.5$, then we would have converged within one iteration. The choice of step sizes that lead to convergence and divergence from a minimizer are classical results in numerical optimization, and are reproduced in the following result.

\begin{theorem}[See \cite{bertsekas1999}, Ch. 1] \label{theorem:gd}
Let $Q \in \mathbb{R}^{p \times p}$ be a symmetric, positive definite matrix; let $\theta^* \in \mathbb{R}^p$ be arbitrary; and define $f(\theta) = 0.5(\theta - \theta^*)'Q(\theta - \theta^*)$. Let $\lbrace \theta_k:k\in\mathbb{N} \rbrace$ be the iterated generated by gradient descent on $f(\theta)$, initialized with an arbitrary starting point $x_0 \in \mathbb{R}^p\setminus \lbrace \theta^* \rbrace$ and a constant step size $C > 0$. 
\begin{enumerate}
\item Let $\lambda_{\max}(Q)$ denote the maximum eigenvalue of $Q$. If $C < 2/\lambda_{\max}(Q)$, then $\norm{\theta_k - \theta^*}_2 \to 0$ as $k \to \infty$.
\item Let $\lambda_{\min}(Q)$ denote the minimum eigenvalue of $Q$. If $C > 2/\lambda_{\min}(Q)$, then $\norm{\theta_k - \theta^*}_2 \to \infty$ as $k \to \infty$.
\end{enumerate}
\end{theorem}

The second observation is that the divergence has a particular pattern: namely, the iterates are diverging at an exponential rate of $|1 - 2C| = 1.2$ from the minimizer.

As we will show, these two observations are preserved as we move from gradient descent to SGD-$k$ (i.e., SGD with a batch size of $k$). First, our analysis of SGD-$k$ will generalize and include as a special case Theorem \ref{theorem:gd}.
Second, our analysis of SGD-$k$ will also preserve the exponential divergence properties observed for gradient descent.
Therefore, when we formulate our deterministic mechanism for nonconvex problems, these two properties will be preserved in its formulation. 

Importantly, in \S \ref{section:homogeneous} and \S \ref{section:inhomogeneous}, we will verify our deterministic mechanism by determining if both of these observations hold in the nonconvex setting. Indeed, if these two properties are verified, then we will supply evidence in support of our determinsitic mechanism. Now, to establish our deterministic mechanism, we turn to generalizing Theorem \ref{theorem:gd} for SGD-$k$ on a generic quadratic problem.

\section{A Deterministic Mechanism} \label{section:stochastic-quadratic}
In classical convergence analysis, a standard technique is to locally analyze the behavior of the optimizer's iterates using a quadratic approximation to the objective function near the minimizer. Motivated by this standard technique, we begin by formulating a generic  stochastic quadratic problem, and derive some of its relevant geometric properties (\S \ref{subsection:stochastic-quadratic}). Then, we define and study the properties of SGD-$k$ on this stochastic quadratic problem to rigorously analyze the impact of geometry and batch size on the convergence and divergence properties of SGD-$k$ (\S \ref{subsection:analysis}). Guided by these results, we formulate our deterministic mechanism for divergence of SGD-$k$ on nonconvex problems, and discuss when the generic stochastic quadratic problem can be applied as a local analysis and when it falls short (\S \ref{subsection:deterministic-mechanism}).

\subsection{A Generic Stochastic Quadratic Problem} \label{subsection:stochastic-quadratic}

We will begin by defining the quadratic problem, derive some of its relevant properties, and discuss the two types of minimizers that can occur for the problem.

\begin{problem}[Quadratic Problem] \label{problem-sgd:quadratic}
Let $(\Omega,\mathcal{F},\mathbb{P})$ be a probability space. Let $Q \in \mathbb{R}^{p \times p}$ be a nonzero, symmetric, positive semidefinite random matrix and $r \in \mathbb{R}^p$ be a random vector in the image space of $Q$ (with probability one) such that (1) $\E{Q} \prec \infty$, (2) $\E{ Q \E{Q} Q} \prec \infty$, (3) $\E{Q \E{Q} r} $ is finite, and (4) $\E{r'Qr}$ is finite. Let $\lbrace (Q_N,r_N): N \in \mathbb{N} \rbrace$ be independent copies of $(Q,r)$. The quadratic problem is to use $\lbrace (Q_N,r_N) : N \in \mathbb{N} \rbrace$ to determine a $\theta^*$ such that
\begin{equation} \label{eqn-sgd-problem:quadratic-obj}
\theta^* \in \underset{\theta \in \mathbb{R}^p}{\argmin} \frac{1}{2}\theta'\E{Q}\theta + \E{r}'\theta.
\end{equation}
\end{problem}

There are several aspects of the formulation of Problem \ref{problem-sgd:quadratic} worth highlighting. First, $Q$ is required to be symmetric and positive semidefinite, and $r$ is required to be in the image space of $Q$. Together, these requirements on $(Q,r)$ ensure that for each $\omega$ in some probability one set in $\mathcal{F}$, there exists a minimizer for $0.5 \theta'Q\theta + r'\theta$. While this condition seems quite natural for the quadratic problem, it will cause some challenges for generalizing our results directly to the nonconvex case, which we will discuss further in \S \ref{subsection:deterministic-mechanism}. Second, we have required a series of assumptions on the moments of $(Q,r)$. As we will see in the analysis of SGD-$k$ on Problem \ref{problem-sgd:quadratic}, these moment assumptions will be essential to the stability of the iterates. Finally, we have not required that the solution, $\theta^*$, be unique. That is, we are not requiring a strongly convex problem. This generality is important as it will allow us to better approximate more exotic nonconvex problems. 

We now establish some basic geometric properties about the quadratic problem. We first establish a standard, more convenient reformulation of the objective function in (\ref{eqn-sgd-problem:quadratic-obj}). Then, we formulate some important geometric properties of the quadratic problem.

\begin{lemma} \label{lemma-sgd:objective-reformulation}
The objective function in (\ref{eqn-sgd-problem:quadratic-obj}) is equivalent to (up to an additive constant)
\begin{equation}
\frac{1}{2}(\theta - \theta^*)'\E{Q}(\theta - \theta^*),
\end{equation}
for any $\theta^*$ satisfying (\ref{eqn-sgd-problem:quadratic-obj}). 
\end{lemma}
\begin{proof}
If $\theta^*$ is a minimizer of the objective in (\ref{eqn-sgd-problem:quadratic-obj}), then $\theta^*$ must be a stationary point of the objective function. This implies that $-\E{Q}\theta^* = \E{r}$. Therefore, the objective in (\ref{eqn-sgd-problem:quadratic-obj}) is, up to an additive constant,
\begin{equation}
\frac{1}{2} \theta'\E{Q}\theta - \theta' \E{Q} \theta^* + \frac{1}{2}(\theta^*)'\E{Q} \theta^* = \frac{1}{2}(\theta - \theta^*)'\E{Q}(\theta - \theta^*),
\end{equation} 
\end{proof}

Because $Q$ is symmetric, positive semidefinite with probability one, $\E{Q}$ will also be symmetric, positive semidefinite. That is, $m := \mrank(\E{Q}) \leq p$. For future reference, we will denote the nonzero eigenvalues of $\E{Q}$ by $\lambda_1 \geq \cdots \lambda_m > 0$. Beyond the eigenvalues, we will also need some higher order curvature information, namely, $s_Q$ and $t_Q$, which are given by
\begin{equation} \label{eqn-sgd:tQ}
t_Q = \sup \left\lbrace \frac{v'\E{Q \E{Q} Q}v - v' \E{Q}^3 v}{v' \E{Q} v} : v \in \mathbb{R}^p, ~ v'\E{Q}v \neq 0 \right\rbrace,
\end{equation}
and
\begin{equation} \label{eqn-sgd:sQ}
s_Q = \inf \left\lbrace \frac{v'\E{Q \E{Q} Q}v - v' \E{Q}^3 v}{v' \E{Q} v} : v \in \mathbb{R}^p, ~ v'\E{Q}v \neq 0 \right\rbrace.
\end{equation}
From the definitions in (\ref{eqn-sgd:tQ}) and (\ref{eqn-sgd:sQ}), it follows that $t_Q \geq s_Q$. The next result establishes that for nontrivial problems, $s_Q > 0$. 

\begin{lemma} \label{lemma-sgd:higher-moment-bounds}
Let $Q$ be as in (\ref{problem-sgd:quadratic}). Then, the following properties hold.
\begin{enumerate}
\item If there is an $x \in \mathbb{R}^p$ such that $x' \E{Q} x = 0$ then $x' \E{ Q \E{Q} Q} x = 0$. 
\item $\E{Q \E{Q} Q} \succeq \E{Q}^3$. 
\item If $\Prb{Q = \E{Q}} < 1$ then $s_Q > 0$. 
\end{enumerate} 
\end{lemma}
\begin{proof}
For (1), when $x=0$ the result follows trivially. Suppose then that there is an $x \neq 0$ such that $x'\E{Q}x = 0$. Then, since $Q \succeq 0$, $x'Qx = 0$ almost surely. Therefore, $x$ is in the null space of all $Q$ on a probability one set. Therefore, $x'Q \E{Q} Q x = 0$ with probability one. The conclusion follows.

For (2), note that for any $x \in \mathbb{R}^q$, the function $f(M) = x'M \E{Q} Mx$ is convex over the space of all symmetric, positive semi-definite matrices. Therefore, by Jensen's inequality, $\E{f(Q)} \geq f(\E{Q})$. Moreover, since $x$ is arbitrary, the conclusion follows.

For (3), note that Jensen's inequality holds with equality only if $f$ is affine or if $\Prb{Q = \E{Q}} = 1$. Therefore, if $\Prb{Q = \E{Q}} < 1$, then the second condition is ruled out. Moreover, if $\Prb{Q = \E{Q}} < 1$, then $\Prb{Q = 0} < 1$. Thus, $\exists x \neq 0$ such that $f(Q) \neq 0$, which rules out the first condition.  
\end{proof}

Now, if $\Prb{Q = \E{Q}} = 1$ then there exists a solution to (\ref{eqn-sgd-problem:quadratic-obj}) such that $Q\theta^* + r = 0$ with probability one. This phenomenon of the existence of a single solution for all pairs $(Q,r)$ with probability one will play a special role, and motivates the following definition.

\begin{definition}[Homogeneous and Inhomogeneous Solutions]
A solution to $\theta^*$ satisfying (\ref{eqn-sgd-problem:quadratic-obj}) is called homogeneous if
\begin{equation} \label{eqn-sgd-definition:homogeneous}
\theta^* \in \argmin_{\theta \in \mathbb{R}^p} \frac{1}{2} \theta'Q \theta + r'\theta \text{ with probability one.}
\end{equation}
Otherwise, the solution is called inhomogeneous.
\end{definition}

In the case of a homogeneous minimizer, the objective function in (\ref{eqn-sgd-definition:homogeneous}) can be rewritten as (up to an additive constant)
\begin{equation}
\frac{1}{2} (\theta - \theta^*)' Q (\theta - \theta^*),
\end{equation}
which follows from the same reasoning Lemma \ref{lemma-sgd:objective-reformulation}. Importantly, SGD-$k$ will behave differently for problems with homogeneous minimizers and inhomogeneous minimizer. As we will show, for homogeneous minimizers, SGD-$k$ behaves rather similarly to classical gradient descent in terms of convergence and divergence, whereas for inhomogeneous minimizers, SGD-$k$ will behave markedly differently. To show these results, we will first define SGD-$k$ for the quadratic problem.

\begin{definition}[SGD-$k$] \label{definition-sgd:sgd-k}
Let $\theta_0 \in \mathbb{R}^p$ be arbitrary. For Problem \ref{problem-sgd:quadratic}, SGD-$k$ generates a sequence of iterates $\lbrace \theta_N : N \in \mathbb{N} \rbrace$ defined by
\begin{equation} \label{eqn-definition:sgd-k}
\theta_{N+1} = \theta_N - \frac{C_{N+1}}{k}\sum_{j=Nk+1}^{N(k+1)} Q_j\theta_N + r_j,
\end{equation} 
where $\lbrace C_N: N \in \mathbb{N} \rbrace$ is a sequence of scalars. 
\end{definition}

In Definition \ref{definition-sgd:sgd-k}, when $k = 1$, we have the usual notion of stochastic gradient descent and, as $k \to \infty$, we recover gradient descent. Therefore, this notion of SGD-$k$ will allow us to generate results that explore the complete range of stochastic to deterministic methods, and bridge the theory between stochastic gradient descent and gradient descent.

\subsection{SGD-$k$ on the Stochastic Quadratic Problem} \label{subsection:analysis}

\begin{table}
\tbl{Summary of Notation}
{
\begin{tabular}{@{}ll@{}} \toprule
Notation & Value \\ \midrule
$M$ & $\E{Q \E{Q} Q} - \E{Q}^3$ \\
$e_N$ & $(\theta_N - \theta^*)'\E{Q}(\theta_N - \theta^*)$ \\
$e_{N,l}$ & $(\theta_N - \theta^*)'\E{Q}^l(\theta_N - \theta^*)$ for $l \in \mathbb{N}_{\geq 2}$ \\
$e_{N,M}$ & $(\theta_N - \theta^*)'M (\theta_N - \theta^*)$ \\ 
$m$ & $\mrank( \E{Q} )$ \\
$\lambda_1 \geq \cdots \geq \lambda_m > 0$ & Nonzero eigenvalues of $\E{Q}$ \\ 
$t_Q, s_Q$ & Higher-order Curvature Parameters of $Q$ \\\bottomrule
\end{tabular}
}
\label{table:summary-of-notation}
\end{table}

Table \ref{table:summary-of-notation} summarizes the notation that will be used throughout the remainder of this work. With this notation, we now analyze the iterates generated by SGD-$k$ for Problem \ref{problem-sgd:quadratic}. Our first step is to establish a relationship between $e_{N+1}$ and $e_N$, which is a standard computation using conditional expectations.

\begin{lemma} \label{lemma-sgd:quadratic-recursion}
Let $\lbrace \theta_N : N \in \mathbb{N} \rbrace$ be the iterates of SGD-$k$ from Definition \ref{definition-sgd:sgd-k}. Let $\theta^*$ be a solution to the quadratic problem. Then, with probability one, 
\begin{equation}
\begin{aligned}
\cond{e_{N+1}}{\theta_N} &= e_{N} - 2C_{N+1}e_{N,2} + C_{N+1}^2 e_{N,3} + \frac{C_{N+1}^2}{k} e_{N,M} \\
						&+ 2\frac{C_{N+1}^2}{k}(\theta_N - \theta^*)' \E{ Q \E{Q} (Q\theta^* + r)} \\
						&+ \frac{C_{N+1}^2}{k} \E{ (Q\theta^* + r)' \E{Q}(Q \theta^* + r)}.
\end{aligned}
\end{equation}
Moreover, if $\theta^*$ is a homogeneous minimizer, then, with probability one,
\begin{equation}
\begin{aligned}
\cond{e_{N+1}}{\theta_N} &= e_{N} - 2C_{N+1} e_{N,2} + C_{N+1}^2 e_{N,3} + \frac{C_{N+1}^2}{k} e_{N,M}.
\end{aligned}
\end{equation}
\end{lemma}
\begin{proof}
The result is a straightforward calculation from the properties of SGD-$k$ and the quadratic problem. In the case of the homogeneous minimizer, recall that $Q\theta^* +r = 0$ with probability one.
\end{proof} 

Our second step is to find bounds on $\cond{e_{N+1}}{\theta_N}$ in terms of $e_N$. While this is a standard thing to do, our approach differs from the standard approach in two ways. First, not only will we find an upper bound, but we will also find a lower bound. Second, our lower bounds will be much more delicate than the usual strong convexity and Lipschitz gradient based arguments \cite[see][\S 4]{bottou2016}. This second step can be accomplished for the homogeneous case with the following two lemmas.

\begin{lemma} \label{lemma-sgd:technical-systematic-component}
In the setting of Lemma \ref{lemma-sgd:quadratic-recursion} and for $\alpha_j \geq 0$ for $j=0,1,2,3$, 
\begin{equation} \label{eqn-sgd-lemma:technical-systematic-component}
\begin{aligned}
&\alpha_0e_N - \alpha_1 C_{N+1} e_{N,2} + \alpha_2 C_{N+1}^2 e_{N,3} + \alpha_3 C_{N+1}^2  e_{N,M}
\end{aligned}
\end{equation}
is bounded above by
\begin{equation}
e_N\left[ \alpha_0 - \alpha_1 C_{N+1} \lambda_j + \alpha_2 C_{N+1}^2 \lambda_j^2 + \alpha_3 C_{N+1}^2 t_Q \right], 
\end{equation}
where
\begin{equation}
j = \begin{cases}
1 & C_{N+1} > \frac{\alpha_1}{\alpha_2} \frac{1}{\lambda_1 + \lambda_m} \text{ or } C_{N+1} \leq 0 \\
m & \text{otherwise}.
\end{cases}
\end{equation}
Moreover, (\ref{eqn-sgd-lemma:technical-systematic-component}) is bounded below by
\begin{equation}
e_N\left[ \alpha_0 - \alpha_1 C_{N+1} \lambda_j + \alpha_2 C_{N+1}^2 \lambda_j^2 + \alpha_3 C_{N+1}^2 s_Q \right],
\end{equation}
where 
\begin{equation}
j = \begin{cases}
1 & C_{N+1} \in \left(0 ,\frac{\alpha_1}{\alpha_2} \frac{1}{\lambda_1 + \lambda_2}\right] \\
l ~(l \in \lbrace 2,\ldots,m-1 \rbrace) & C_{N+1} \in \left( \frac{\alpha_1}{\alpha_2} \frac{1}{\lambda_{l-1} + \lambda_l}, \frac{\alpha_1}{\alpha_2} \frac{1}{\lambda_l + \lambda_{l+1}} \right] \\
m & C_{N+1} > \frac{\alpha_1}{\alpha_2} \frac{1}{\lambda_{m-1} + \lambda_m} \text{ or } C_{N+1} \leq 0.
\end{cases}
\end{equation}
\end{lemma}
\begin{proof}
When $e_N = 0 $, then $e_{N,j} =0$ for $j=2,3$ and $e_{N,M}=0$ by Lemma \ref{lemma-sgd:higher-moment-bounds}. Therefore, if $e_N = 0$, the bounds hold trivially. So, we will assume that $e_N > 0$. 

Let $u_1,\ldots,u_m$ be the orthonormal eigenvectors corresponding to eigenvalues $\lambda_1,\ldots,\lambda_m$ of $\E{Q}$. Then, 
\begin{equation}
\frac{e_{N,j}}{e_N} = \sum_{l=1}^m \lambda_l^{j-1}\frac{\lambda_l (u_l'(\theta_N - \theta^*))^2}{e_N}.
\end{equation}
Denote the ratio on the right hand side by $w_l$, and note that $\lbrace w_l :l=1,\ldots,m\rbrace$ sum to one and are nonnegative. With this notation, we bound (\ref{eqn-sgd-lemma:technical-systematic-component}) from above by
\begin{equation}
e_N\left[ \alpha_0 - \alpha_1 C_{N+1} \sum_{l=1}^m \lambda_l w_l + \alpha_2 C_{N+1}^2 \sum_{l=1}^m \lambda_l^2 w_l + \alpha_3 C_{N+1}^2 t_Q \right],  
\end{equation}
and from below by the same equation but with $t_Q$ replaced by $s_Q$. By the properties of $\lbrace w_l : l=1,\ldots, m \rbrace$, we see that the bounds are composed of convex combinations of quadratics of the eigenvalues. Thus, for the upper bound, if we assign all of the weight to the eigenvalue that maximizes the polynomial $-\alpha_1 C_{N+1} \lambda + \alpha_2 C_{N+1}^2 \lambda^2$, then we will have the upper bound presented in the result. To compute the lower bounds, we do the analogous calculation.
\end{proof}

Using these two lemmas, we can conclude the following for the homogeneous case.

\begin{theorem}[Homogeneous Minimizer] \label{theorem-sgd:homogeneous}
Let $\lbrace \theta_N : N \in \mathbb{N} \rbrace$ be the iterates of SGD-$k$ from Definition \ref{definition-sgd:sgd-k}. Let $\theta^*$ be a homogeneous solution to the quadratic problem.

If $0 < C_{N+1}$ and 
\begin{equation} \label{eqn-sgd-theorem:homogeneous-ub}
C_{N+1} < \begin{cases}
\frac{2\lambda_1}{\lambda_1^2 + t_Q/k} & k > t_Q/(\lambda_1 \lambda_m) \\
\frac{2\lambda_m}{\lambda_m^2 + t_Q/k} & k \leq t_Q/(\lambda_1 \lambda_m),
\end{cases}
\end{equation}
then $\E{e_{N+1}} < \E{e_N}$. Moreover, if $\lbrace C_{N+1} \rbrace$ satisfy (\ref{eqn-sgd-theorem:homogeneous-ub}) and are uniformly bounded away from zero, then $\E{e_{N}}$ decays exponentially to zero.

Moreover, there exists a deterministic constant $\rho_N > 1$ such that $\cond{e_{N+1}}{\theta_N} > \rho_N e_N$ with probability one, if either $C_{N+1} < 0$ or 
\begin{equation} \label{eqn-sgd-theorem:homogeneous-lb}
C_{N+1} > \frac{2\lambda_j}{\lambda_j^2 + s_Q/k},
\end{equation}
where
\begin{equation}
j = \begin{cases}
1 & k \leq \frac{s_Q}{\lambda_1\lambda_2} \\
l ~(l \in \lbrace 2,\ldots,m-1 \rbrace) & \frac{s_Q}{\lambda_l\lambda_{l-1}} < k \leq  \frac{s_Q}{\lambda_{l+1}\lambda_l} \\
m & k > \frac{s_Q}{\lambda_m \lambda_{m-1}}.
\end{cases}
\end{equation}
\end{theorem}
\begin{proof}
The upper and lower bounds follow from Lemma \ref{lemma-sgd:quadratic-recursion} and Lemma \ref{lemma-sgd:technical-systematic-component}, and observations that $k > t_Q/(\lambda_1 \lambda_m)$ implies
\begin{equation} \label{eqn-sgd-theorem:homogeneous-proof-1}
\frac{2\lambda_j}{\lambda_j^2 + t_Q/k} > \frac{2}{\lambda_1 + \lambda_m},
\end{equation}
for $j = 1,m$. For convergence, when $C_{N+1} > 0$ and strictly less than expression on the left hand side of (\ref{eqn-sgd-theorem:homogeneous-proof-1}), the upper bound in Lemma \ref{lemma-sgd:technical-systematic-component} is strictly less than one. Thus, if $\lbrace C_{N+1} \rbrace$ are uniformly bounded away from zero, the exponential convergence rate of $\E{e_N}$ follows trivially.

For the lower bound, we make note of several facts. First, if $k \geq \frac{s_Q}{\lambda_l \lambda_{l+1}}$ then $k \geq \frac{s_Q}{\lambda_j \lambda_{j+1}}$ for $j=1,2,\ldots,l$. Moreover, when $k \geq \frac{s_Q}{\lambda_l \lambda_{l+1}}$, then 
\begin{equation}
\frac{2\lambda_l}{\lambda_l^2 + s_Q/k} \geq \frac{2}{\lambda_l + \lambda_{l+1}}.
\end{equation}
By Lemma \ref{lemma-sgd:technical-systematic-component}, the left hand side of this inequality is the lower bound on $C_{N+1}$ that guarantees that $\E{e_{N+1}} > e_N$ with probability one. Therefore, whenever $k \geq \frac{s_Q}{\lambda_l\lambda_{l+1}}$, $\cond{e_N+1}{\theta_N}$ is not guaranteed to diverge. On the other hand, if $k < \frac{s_Q}{\lambda_l \lambda_{l-1}}$ then $k < \frac{s_Q}{\lambda_j \lambda_{j-1}}$ for $j=l,l+1,\ldots,m$. Moreover, if $k < \frac{s_Q}{\lambda_l \lambda_{l-1}}$ then
\begin{equation}
\frac{2\lambda_l}{\lambda_l^2 + s_Q/k} < \frac{2}{\lambda_l + \lambda_{l-1}}.
\end{equation}
Therefore, since the left hand side of this inequality is the lower bound on $C_{N+1}$ that guarantees that there exists a deterministic scalar $ \rho_N > 1$ such that $\cond{e_{N+1}}{\theta_N} > \rho_N e_N$ with probability one, the lower bound is guaranteed to diverge for $C_{N+1}$ larger than the right hand side.

Thus, by the monotonicty of the eigenvalues, for the $l \in \lbrace 2,\ldots,m-1 \rbrace$ such that $\frac{s_Q}{\lambda_l \lambda_{l-1}} < k \leq \frac{s_Q}{\lambda_l\lambda_{l+1}}$, $\cond{e_{N+1}}{\theta_N} > \rho_N e_N$ if
\begin{equation}
C_{N+1} > \frac{2\lambda_l}{\lambda_l^2 + s_Q/k}.
\end{equation}
Note, we can handle the edge cases similarly.
\end{proof}
\begin{remark}
Because $\rho_N$ is a deterministic scalar quantity, we can conclude that $\E{e_{N+1}} > \rho_N \E{e_N}$. 
\end{remark}

Using Theorem \ref{theorem-sgd:homogeneous}, as $k \to \infty$, we recover Theorem \ref{theorem:gd}. That is, Theorem \ref{theorem-sgd:homogeneous} contains gradient descent as a special case. Moreover, Theorem \ref{theorem-sgd:homogeneous} also captures the exponential divergence property that we observed for gradient descent (see Figure \ref{figure:divergence-gd}). Theorem \ref{theorem-sgd:homogeneous} also contains additional subtleties. First, it captures SGD-$k$'s distinct phases for convergence and divergence, which depend on the batch size and expected geometry of the quadratic problem as represented by the eigenvalues, $t_Q$ and $s_Q$. Thus, as opposed to classical gradient descent, SGD-$k$ requires a more complex convergence and divergence statement to capture these distinct phases. 

We now state the analogous statement for the inhomogeneous problem. However, we will need the following additional lemmas to bound the behavior of the cross term in Lemma \ref{lemma-sgd:quadratic-recursion}. Again, our analysis below is more delicate than the standard approach for analyzing SGD. 

\begin{lemma} \label{lemma-sgd:technical-pseudo-inverse}
Let $Q$ and $r$ be as in Problem \ref{problem-sgd:quadratic}. Then, letting $(\cdot)^\dagger$ denote the Moore-Penrose pseudo-inverse, 
\begin{equation}
\E{ Q \E{Q}(Q \theta^* + r )} = \E{Q}\E{Q}^\dagger\E{ Q \E{Q}(Q \theta^* + r )} .
\end{equation}
\end{lemma}
\begin{proof}
Note that for any $x \in \mrow(\E{Q})$, $x = \E{Q}\E{Q}^\dagger x$. Thus, we must show that $\E{Q \E{Q} (Q \theta^* + r)}$ is in $\mrow(\E{Q})$. Recall, that if there exists a $v \in \mathbb{R}^p$ such that $\E{Q}v = 0$, then $v' \E{Q} v = 0$. In turn, $v'Qv = 0$ with probability one, which implies that $Q v = 0$ with probability one. Hence, $\mnull(\E{Q}) \subset \mnull( Q)$ with probability one. Let the set of probability one be denoted by $\Omega'$. Then, 
\begin{equation}
\mrow( \E{Q}) = \mnull(\E{Q})^\perp \supset \bigcup_{\omega \in \Omega'} \mnull(Q)^\perp = \bigcup_{\omega \in \Omega'} \mcol( Q) \supset \bigcup_{\omega \in \Omega'} \mcol( Q \E{Q} ),
\end{equation}
where $\mrow(Q) = \mcol(Q)$ by symmetry. Therefore, $Q \E{Q} (Q\theta^* + r) \in \mrow(\E{Q})$ with probability one. Hence, its expectation is in $\mrow(\E{Q})$. 
\end{proof}

\begin{lemma} \label{lemma-sgd:technical-cross-term-bound}
Under the setting of Lemma \ref{lemma-sgd:quadratic-recursion}, for any $\phi > 0$ and $j \in \mathbb{N}$,  
\begin{equation}
\begin{aligned}
&2 \left\vert\frac{C_{N+1}^2}{k} (\theta_N - \theta^*)' \E{Q \E{Q}(Q\theta^* + r)}\right\vert \\
&\quad \leq 2 \frac{C_{N+1}^2}{k} \norm{ \E{Q}^{j/2} (\theta_N - \theta^*)}_2 \norm{(\E{Q}^{j/2})^\dagger \E{Q \E{Q}(Q\theta^* + r)}}_2 \\
&\quad \leq \frac{\phi C_{N+1}^2}{k} (\theta_N - \theta^*)' \E{Q}^j (\theta_N - \theta^*) \\
&\quad \quad + \frac{C_{N+1}^2}{\phi k}\E{(Q\theta^* + r)' \E{Q} Q} ( \E{Q}^j)^\dagger \E{Q \E{Q}(Q \theta^* + r)},
\end{aligned}
\end{equation}
for any $N = 0,1,2,\ldots$.
\end{lemma}
\begin{proof}
We will make use of three facts: Lemma \ref{lemma-sgd:technical-pseudo-inverse}, H\"{o}lder's inequality, and the fact that for any $\phi > 0$, $ \varphi \geq 0$ and $x \in \mathbb{R}$, $2 |\varphi x| \leq \phi x^2 + \phi^{-1} \varphi^2.$ Using these facts, we have
\begin{equation}
\begin{aligned}
&2 \left\vert\frac{C_{N+1}^2}{k} (\theta_N - \theta^*)' \E{Q \E{Q}(Q\theta^* + r)}\right\vert \\
&\quad = 2 \left\vert\frac{C_{N+1}^2}{k} (\theta_N - \theta^*)'\E{Q}^{j/2}(\E{Q}^{j/2})^\dagger \E{Q \E{Q}(Q\theta^* + r)}\right\vert\\
&\quad \leq 2 \frac{C_{N+1}^2}{k} \norm{ \E{Q}^{j/2} (\theta_N - \theta^*)}_2 \norm{(\E{Q}^{j/2})^\dagger \E{Q \E{Q}(Q\theta^* + r)}}_2 \\
&\quad \leq \frac{\phi C_{N+1}^2}{k} (\theta_N - \theta^*)' \E{Q}^j (\theta_N - \theta^*) \\
&\quad +   \frac{C_{N+1}^2}{\phi k}\E{(Q\theta^* + r)' \E{Q} Q} ( \E{Q}^j)^\dagger \E{Q \E{Q}(Q \theta^* + r)}.
\end{aligned}
\end{equation}
\end{proof}

\begin{lemma} \label{lemma-sgd:quadratic-recursion-ineq}
Under the setting of Lemma \ref{lemma-sgd:quadratic-recursion},
\begin{equation}
\begin{aligned}
\cond{e_{N+1}}{\theta_N} &\leq e_{N} - 2C_{N+1} e_{N,2} + C_{N+1}^2 e_{N,3} + \frac{C_{N+1}^2}{k} e_{N,M} + \frac{C_{N+1}^2}{k} e_{N,3} \\
						&+ \frac{C_{N+1}^2}{k} \E{ (Q\theta^* + r)' \E{Q}(Q \theta^* + r)} \\
						&+ \frac{C_{N+1}^2}{k} \E{(Q\theta^*+r)'\E{Q} Q} ( \E{Q}^3)^\dagger \E{Q \E{Q}(Q \theta^* + r)}
\end{aligned}
\end{equation}
Now, for any $\varphi > 0$, $\exists \psi \in \mathbb{R}$ such that 
\begin{equation}
\begin{aligned}
\cond{e_{N+1}}{\theta_N} &\geq \left(1 - \frac{\varphi}{k} \right) e_{N} - 2C_{N+1}e_{N,2}  + C_{N+1}^2 e_{N,3} + \frac{C_{N+1}^2}{k} e_{N,M} \\
						&+ \frac{C_{N+1}^2}{k} \psi,
\end{aligned}
\end{equation}
where $\psi \geq 0$ if
\begin{equation}
\varphi\frac{\E{ (Q\theta^* + r)' \E{Q}(Q \theta^* + r)}}{\E{(Q\theta^*+r)'\E{Q} Q} \E{Q}^\dagger \E{Q \E{Q}(Q \theta^* + r)}} \geq C_{N+1}^2.
\end{equation}
\end{lemma}
\begin{proof}
When $C_{N+1} = 0,$ the statements hold by Lemma \ref{lemma-sgd:quadratic-recursion}. Suppose $C_{N+1} \neq 0$. Using Lemma \ref{lemma-sgd:quadratic-recursion} and applying Lemma \ref{lemma-sgd:technical-cross-term-bound} with $\phi =1$ and $j=3$, the upper bound follows. Now, using Lemma \ref{lemma-sgd:quadratic-recursion} and applying Lemma \ref{lemma-sgd:technical-cross-term-bound} with $\phi =  C_{N+1}^{-2} \varphi$ for $\varphi > 0$ and $j=1$, we have
\begin{equation}
\begin{aligned}
\cond{e_{N+1}}{\theta_N} &\geq \left(1 - \frac{\varphi}{k} \right) e_{N} - 2C_{N+1} e_{N,2} + C_{N+1}^2 e_{N,3} + \frac{C_{N+1}^2}{k} e_{N,M}  \\
						&+ \frac{C_{N+1}^2}{k} \psi,
\end{aligned}
\end{equation}
where
\begin{equation}
\begin{aligned}
\psi &=  - \frac{C_{N+1}^2}{\varphi}  \E{(Q\theta^*+r)'\E{Q} Q}\E{Q}^\dagger \E{Q \E{Q}(Q \theta^* + r)} \\
&\quad + \E{ (Q\theta^* + r)' \E{Q}(Q \theta^* + r)}.
\end{aligned}
\end{equation}
Therefore, if $\varphi$ is given as selected in the statement of the result, then $\psi \geq 0$.
\end{proof}

\begin{theorem}[Inhomogeneous Minimizer] \label{theorem-sgd:inhomogeneous}
Let $\lbrace \theta_N : N \in \mathbb{N} \rbrace$ be the iterates of SGD-$k$ from Definition \ref{definition-sgd:sgd-k}.  Suppose $\theta^*$ is an inhomogeneous solution to the quadratic problem.

If $0 < C_{N+1}$ and
\begin{equation} \label{eqn-theorem-sgd:inhomogeneous-ub}
C_{N+1} < \begin{cases}
\frac{2\lambda_1^2}{(1 + 1/k) \lambda_1^2 + t_Q/k } & k+1 > t_Q/(\lambda_1 \lambda_m) \\
\frac{2\lambda_m^2}{(1 + 1/k) \lambda_m^2 + t_Q/k } & k+1 \leq t_Q/(\lambda_1\lambda_m), 
\end{cases}
\end{equation}
then $\E{e_{N+1}} < \rho_N \E{e_{N}} + C_{N+1}^2 \frac{\psi}{k}$, where $\rho_N < 1$ and is uniformly bounded away from zero for all $N$, and $\psi > 0$. Moreover, if $\lbrace C_{N} : N \in \mathbb{N} \rbrace$ are nonnegative, converge to zero and sum to infinity then $\E{e_{N}} \to 0$ as $N \to \infty$.

Furthermore, suppose $C_{N+1} < 0$ or 
\begin{equation} \label{eqn-theorem-sgd:inhomogeneous-lb}
C_{N+1} > \frac{2(\lambda_j + \gamma)}{\lambda_j^2 + s_Q/k},
\end{equation} 
where $4\gamma^2 \in \left(0, \frac{s_Q}{k} \right]$
and 
\begin{equation}
j = \begin{cases}
1 & k \leq \frac{s_Q}{ \lambda_1 \lambda_2 + \gamma(\lambda_1 + \lambda_2)} \\
l ~(l \in \lbrace 2,\ldots,m-1 \rbrace) & \frac{s_Q}{\lambda_l\lambda_{l-1} + \gamma(\lambda_l + \lambda_{l-1})} < k \leq \frac{s_Q}{\lambda_l \lambda_{l+1} + \gamma (\lambda_l + \lambda_{l+1})} \\
m & k > \frac{s_Q}{\lambda_m \lambda_{m-1} + \gamma(\lambda_m + \lambda_{m-1})}.
\end{cases}
\end{equation}
Then, $\exists \delta > 0$ such that if $\norm{\E{Q}(\theta_{N}-\theta^*)} < \delta$ then $\cond{e_{N+1}}{\theta_N} \geq \rho_N e_{N} + C_{N+1}^2 \psi$ with probability one, where $\rho_N > 1$ and $\psi > 0$ independently of $N$.
\end{theorem}
\begin{proof}
For the upper bound, we apply Lemmas \ref{lemma-sgd:quadratic-recursion-ineq} and \ref{lemma-sgd:technical-systematic-component} to conclude that
\begin{equation}
\cond{e_{N+1}}{\theta_N} \leq e_{N}\left(1 - 2 C_{N+1}\lambda_j + \left(1 + \frac{1}{k}\right) C_{N+1}^2\lambda_j^2 + C_{N+1}^2 \frac{t_Q}{k} \right) + C_{N+1}^2 \frac{\psi}{k},
\end{equation}
where $\psi > 0$ is independent of $N$ and
\begin{equation}
j = \begin{cases}
1 & C_{N+1} \leq 0 \text{ or } \frac{2}{(1 + 1/k)(\lambda_1 + \lambda_m)} \\
m & \text{otherwise}.
\end{cases}
\end{equation}
Moreover, the component multiplying $e_N$ is less than one if $C_{N+1} > 0$ and
\begin{equation} \label{eqn-sgd-theorem:proof-inhomo-1}
C_{N+1} < \frac{2\lambda_j}{\left( 1 + \frac{1}{k}\right)\lambda_j^2 + \frac{t_Q}{k}}.
\end{equation}
When $k + 1 > t_Q/(\lambda_1\lambda_m)$, the right hand side of (\ref{eqn-sgd-theorem:proof-inhomo-1}) is larger than $\frac{2}{(1 + 1/k)(\lambda_1 + \lambda_m)}$. The upper bound follows. Convergence follows by applying Lemma 1.5.2 of \cite{bertsekas1999}.

For the lower bound, note that since $\theta^*$ is inhomogeneous, $\Prb{Q = \E{Q}} < 1$. Therefore, $s_Q > 0$. Now, we apply Lemmas \ref{lemma-sgd:quadratic-recursion-ineq} and \ref{lemma-sgd:technical-systematic-component} to conclude that for any $\gamma > 0$ there exists $\psi_N$ such that
\begin{equation}
\cond{e_{N+1}}{\theta_N} \geq e_{N}\left(\left(1 - \frac{4\gamma^2}{\lambda_j^2 + \frac{1}{k}s_Q} \right) - 2 C_{N+1}\lambda_j + C_{N+1}^2\lambda_j^2 + C_{N+1}^2 \frac{s_Q}{k} \right) + C_{N+1}^2 \frac{\psi_N}{k},
\end{equation}
where $\psi_N$ are nonnegative and uniformly bounded from zero for $C_{N+1}$ sufficiently small, and
\begin{equation}
j = \begin{cases}
1 & C_{N+1} \in \left(0, \frac{2}{\lambda_1 + \lambda_2} \right] \\
l ~(l \in \lbrace 2,\ldots,m-1 \rbrace) & C_{N+1} \in \left( \frac{2}{\lambda_{l-1} + \lambda_l}, \frac{2}{\lambda_l + \lambda_{l+1}} \right] \\
m & C_{N+1} \leq 0 \text{ or } \frac{2}{\lambda_m + \lambda_{m-1}} < C_{N+1}.
\end{cases}
\end{equation}
The term multiplying $e_N$ can be rewritten as
\begin{equation} \label{eqn-sgd-theorem:proof-inhomo-2}
1 - \frac{4\gamma^2}{\lambda_j^2 + \frac{1}{k}s_Q} - \frac{\lambda_j^2}{\lambda_j^2 + \frac{1}{k}s_Q} + \left(\lambda_j^2 + \frac{1}{k}s_Q\right)\left( C_{N+1} - \frac{\lambda_j}{\lambda_j^2 + \frac{1}{k}s_Q} \right)^2,
\end{equation}
which nonnegative when $4\gamma^2$ is in the interval given by in the statement of the result. Moreover, if
\begin{equation} \label{eqn-sgd-theorem:proof-inhomo-3}
C_{N+1} > \frac{2(\lambda_j + \gamma)}{\lambda_j^2 + \frac{s_Q}{k}},
\end{equation}
then
\begin{equation}
C_{N+1}	> \frac{\lambda_j}{\lambda_j^2 + \frac{s_Q}{k}} + \frac{\sqrt{\lambda_j^2 + 4\gamma^2}}{\lambda_j^2 + \frac{s_Q}{k}},	
\end{equation}
which implies
\begin{equation}
\left(C_{N+1} - \frac{\lambda_j}{\lambda_j^2 + \frac{s_Q}{k}} \right)^2 > \frac{\lambda_j^2 + 4\gamma^2}{\left(\lambda_j^2 + \frac{s_Q}{k}\right)^2}.
\end{equation}
From this inequality, we conclude that if (\ref{eqn-sgd-theorem:proof-inhomo-3}) holds then (\ref{eqn-sgd-theorem:proof-inhomo-2}) is strictly larger than $1$. Lastly, for any $\tilde{j} \neq j$, 
\begin{equation}
\frac{s_Q}{\lambda_j \lambda_{\tilde{j}} + \gamma(\lambda_j + \lambda_{\tilde{j}})} > k
\end{equation}
is equivalent to
\begin{equation}
\frac{2}{\lambda_j + \lambda_{\tilde{j}}} > \frac{2(\lambda_j + \gamma)}{\lambda_j^2 + \frac{1}{k}s_Q}.
\end{equation}
With these facts, the conclusion for the lower bound follows just as in the proof of Theorem \ref{theorem-sgd:homogeneous}.
\end{proof}

Comparing Theorems \ref{theorem-sgd:homogeneous} and \ref{theorem-sgd:inhomogeneous}, we see that inhomogeneity exacts an additional price. To be specific, call the term multiplying $e_N$ in the bounds as the systematic component and the additive term as the variance component. First, the inhomogeneous case contains a variance component that was not present for the homogeneous case. Of course, this variance component is induced by the inhomogeneity: at each update, SGD-$k$ is minimizing an approximation to the quadratic problem that has a distinct solution from $\theta^*$ with nonzero probability, which requires decaying the step sizes to remove the variability of solutions induced by these approximations. Second, the inhomogeneity shrinks the threshold for step sizes that ensure a reduction in the systematic component, and inflates the threshold for step sizes that ensure an increase in the systematic component.

Continuing to compare the results for the homogeneous and inhomogeneous cases, we also observe that the systematic component diverges exponentially when the expected gradient is sufficiently small (less than $\delta$). Indeed, this is the precise region in which such a result is important: it says that for step sizes that are too large, any estimate near the solution set is highly unstable. In other words, if we initialize SGD-$k$ near the solution set, our iterates will diverge exponentially fast from this region for $C_N$ above the divergence threshold.

To summarize, our analysis of the stochastic quadratic problem defined in Problem \ref{problem-sgd:quadratic} is a generalization of the standard quadratic problem considered in classical numerical optimization. Moreover, our analysis of SGD-$k$ (see Definition \ref{definition-sgd:sgd-k}) on the stochastic quadratic problem generalizes the results and properties observed for gradient descent. In particular, our analysis (see Theorems \ref{theorem-sgd:homogeneous} and \ref{theorem-sgd:inhomogeneous}) derive a threshold for step sizes above which the iterates will diverge, and this threshold includes the threshold for classical gradient descent. Moreover, our analysis predicts an exponential divergence of the iterates from the minimizer.  

\subsection{The Deterministic Mechanism} \label{subsection:deterministic-mechanism}

We now use the insight developed from our analysis of the generic quadratic problem to formulate a deterministic mechanism for how SGD-$k$ ``escapes'' local minimizers. We will begin by stating the generic nonconvex problem, formalizing our deterministic mechanism, and discussing its implications for nonconvex optimization. We then briefly discuss the pitfalls of naive attempts to directly apply the theory developed in \S \ref{subsection:analysis} to arbitrary nonconvex problems.

\begin{problem}[Nonconvex Problem] \label{problem:nonconvex}
Let $(\Omega,\mathcal{F},\mathbb{P})$ be a probability space, and 
let $C^2(\mathbb{R}^p,\mathbb{R})$ be the set of twice continuously differentiable functions from $\mathbb{R}^p$ to 
Moreover, let $f$ be a random quantity taking values in $C^2(\mathbb{R}^p,\mathbb{R})$. Moreover, suppose that
\begin{enumerate}
\item $\E{f(\theta)}$ exists for all $\theta \in \mathbb{R}^p$, and is denoted by $F(\theta)$.
\item Suppose $F$ is bounded from below. 
\item Suppose that $F \in C^2(\mathbb{R}^p,\mathbb{R})$, and $\nabla F(\theta) = \E{ \nabla f(\theta) }$ and $\nabla^2 F(\theta) = \E{ \nabla^2 f(\theta)}$. 
\end{enumerate}
Let $\lbrace f_N: N \in \mathbb{N} \rbrace$ be independent copies of $f$. The nonconvex problem is to use $\lbrace f_N : N \in \mathbb{N} \rbrace$ to determine a $\theta^*$ that is a minimizer of $F(\theta)$. 
\end{problem}

Note, while Problem \ref{problem:nonconvex} focuses on twice differentiable functions, we often only have continuous functions in machine learning applications. However, for most machine learning applications, the continuous functions are only nondifferentiable at finitely many points, which can be approximated using twice continuously differentiable functions.

Before stating the deterministic mechanism, we will also need to extend our definition of SGD-$k$ to the generic nonconvex problem. Of course, this only requires replacing (\ref{eqn-definition:sgd-k}) with 
\begin{equation}
\theta_{N+1} = \theta_N - \frac{C_{N+1}}{k}\sum_{j=Nk+1}^{N(k+1)} \nabla f_j(\theta_N).
\end{equation}
We now state our deterministic mechanism. In this statement, $B(\theta,\epsilon)$ denotes a ball around a point $\theta \in \mathbb{R}^p$ of radius $\epsilon > 0$.

\begin{definition}[Deterministic Mechanism] \label{definition:deterministic mechanism}
Let $\lbrace \theta_N : N \in \mathbb{N} \rbrace$ be the iterates generated by SGD-$k$ with deterministic step sizes $\lbrace C_N : N \in \mathbb{N} \rbrace$ for the nonconvex problem defined in Problem \ref{problem:nonconvex}. Suppose $\theta^*$ is a local minimizer of Problem \ref{problem:nonconvex}.

For an $\epsilon > 0$ sufficiently small, let $\lambda_1 \geq \cdots \geq \lambda_m > 0$ denote the nonzero eigenvalues of
\begin{equation} 
\frac{1}{|B(\theta^*,\epsilon)|} \int_{B(\theta^*,\epsilon)}\nabla^2 F(\theta) d\theta,
\end{equation}
and define $s_f$ to be the maximum between $0$ and 
\begin{equation}
\frac{1}{|B(\theta^*,\epsilon)|}\int_{B(\theta^*,\epsilon)}\inf_{v'\nabla^2 F(\theta) v \neq 0
} \left\lbrace \frac{v' \left(\E{ \nabla^2 f(\theta) \nabla^2 F(\theta) \nabla^2 f(\theta)} - (\nabla^2 F(\theta) )^3\right) v}{v' \nabla^2 F(\theta) v}  \right\rbrace d\theta.
\end{equation}
There is a $\delta > 0$ such that if $\theta_N \in B(\theta^*,\delta) \setminus \lbrace \theta^* \rbrace$ and $C_{N+1} < 0$ or
\begin{equation}
C_{N+1} > \frac{2\lambda_j}{\lambda_j^2 + \frac{1}{k}s_f},
\end{equation} 
where
\begin{equation}
j = \begin{cases}
1 & k \leq \frac{s_f}{ \lambda_1 \lambda_2 } \\
l ~(l \in \lbrace 2,\ldots,m-1 \rbrace) & \frac{s_f}{\lambda_l\lambda_{l-1}} < k \leq \frac{s_f}{\lambda_l \lambda_{l+1}} \\
m & k > \frac{s_f}{\lambda_m \lambda_{m-1}},
\end{cases}
\end{equation}
then for some $\rho_N > 1$, $\cond{F(\theta_{N+1}) - F(\theta^*)}{\theta_N} > \rho_N ( F(\theta_N) - F(\theta^*) )$ with probability one.
\end{definition}

For nonconvex problems, our deterministic mechanism makes several predictions. First, suppose a minimizer's local geometry has several large eigenvalues. Then, for $k$ sufficiently small, SGD-$k$'s divergence from a local minimizer will be determined by these large eigenvalues. This prediction is indeed observed in practice: \cite{keskar2016} noted that methods that were ``more stochastic'' (i.e., small $k$) diverged from regions that had only a fraction of relatively large eigenvalues. 

Second, suppose for a given nonconvex function, we have finitely many local minimizers and, for each $k$, we are able to compute the threshold for each local minimizer (for some sufficiently small $\epsilon > 0$, which may differ for each minimizer). Then, for identical choices in step sizes, $\lbrace C_{N+1} : N \in \mathbb{N} \rbrace$, the number of minimizers for which SGD-$k$ can converge to is a nondecreasing function of $k$. Moreover, SGD-$k$ and Gradient Descent will only converge to minimizers that can support the step sizes $\lbrace C_{N+1} : N \in \mathbb{N} \rbrace$, which implies that SGD-$k$ will prefer flatter minimizers in comparison to gradient descent. Again, this prediction is indeed observed in practice: the main conclusion of \cite{keskar2016} is that SGD-$k$ with smaller $k$ converges to flatter minimizers in comparison to SGD-$k$ with large $k$. 

Third, our deterministic mechanism states that whenever the iterates are within some $\delta$ ball of $\theta^*$ with step sizes that are too large, we will observe exponential divergence. Demonstrating that exponential divergence does indeed occur in practice is the subject of \S \ref{section:homogeneous} and \S \ref{section:inhomogeneous}.

However, before discussing these experiments, it is worth noting the challenges of generalizing the quadratic analysis to the nonconvex case. In terms of convergence, recall that a key ingredient in the local quadratic analysis for classical numerical optimization is stability: once iterates enter a basin of the minimizer, then the iterates will remain within this particular basin of the minimizer. However, for stochastic problems, this will not hold with probability one, especially when the gradient has a noise component with an unbounded support (e.g., linear regression with normally distributed errors). 

Fortunately, attempts to generalize the divergence results will not require this stability property because divergence implies that the iterates will eventually escape from any such basin. The challenge for divergence, however, follows from the fact that the lower bounds are highly delicate, and any quadratic approximation will require somewhat unrealistic conditions on the remainder term in Taylor's theorem for the nonconvex problem. Thus, generalizing the quadratic case to the nonconvex case is something that we are actively pursuing. 

\section{Numerical Study of a Homogeneous Nonconvex Problem} \label{section:homogeneous}
Here we will study the predictions of our deterministic mechanism on a homogeneous nonconvex problem. The nonconvex objective function that we will use is a summation of $N$ functions described in Appendix \ref{section:qc} with different parameters for the quadratic and circular components. Importantly, this nonconvex problem has only two minimizers --- one corresponding to the quadratic component and the other to the circular component---that can be determined analytically, and both of these minimizers are homogeneous. This problem is described presently.

\begin{problem}[Quadratic-Circle Sums] \label{problem-sgd:qcs}
Let $g_1,\ldots,g_N$ be functions as specified (\ref{eqn-sgd:quad-basin}) and $h_1,\ldots,h_N$ be functions as specified by (\ref{eqn-sgd:circ-basin}). The quadratic-circle sums objective function is
\begin{equation} \label{eqn-sgd:qcs-obj}
\sum_{i=1}^N p_i\min \lbrace  g_i(x), h_i(x) \rbrace,
\end{equation}
where $p_i$ are positive valued and sum to one. Let $Y$ be a random variable taking values $\lbrace (g_i,h_i): i =1,\ldots,N \rbrace$ with the probability of sampling $(g_i,h_i)$ equals to $p_i$. Let $Y_1,Y_2,\ldots$ be independent copies of $Y$. The quadratic-circle sums problem is to use $\lbrace Y_1,Y_2,\ldots \rbrace$ to minimize (\ref{eqn-sgd:qcs-obj}) restricted to $(-10,10) \times (-20, 15) \subset \mathbb{R}^{2}$. 
\end{problem}

Again, we will use this problem to either verify or negate our deterministic mechanism. We will focus on the following three questions.
\begin{enumerate}
\item Is our predicted threshold for divergence valid?
\item If we analogously compute a threshold for convergence, will it be valid?
\item In the case of divergence, do we observe exponential divergence?
\end{enumerate}

We will design four model quadratic-circular sums problems on which to analyze these questions. All of the model functions have two minimizers in the region of interest: one for the circle-basin component at $(0,0)$ and one for quadratic basin component at $(0,-15)$. However, the sharpness of the basins about each minimizers distinguishes the four model quadratic-circle sums problems. The first model, Model 1, is characterized by having at least one large eigenvalue about each minimizer. The second model, Model 2, is characterized by having at least one large eigenvalue for the circular basin minimizer and only small eigenvalues for the quadratic basin minimizer. Model 3 is characterized by having at least one large eigenvalue for the quadratic basin minimizer and only small eigenvalues for the circular basin minimizer. Model 4 is characterized by having only small eigenvalues for both minimizers. Note, the size of these eigenvalues are relative to the other models and not some general baseline. For clarity, each of the four model expected quadratic-circle sums objective functions are visualized in Figure \ref{figure-sgd:qc-models}.

\begin{figure}[hbt]
\centering
\subfloat[Model 1]{\includegraphics[width=0.45\textwidth]{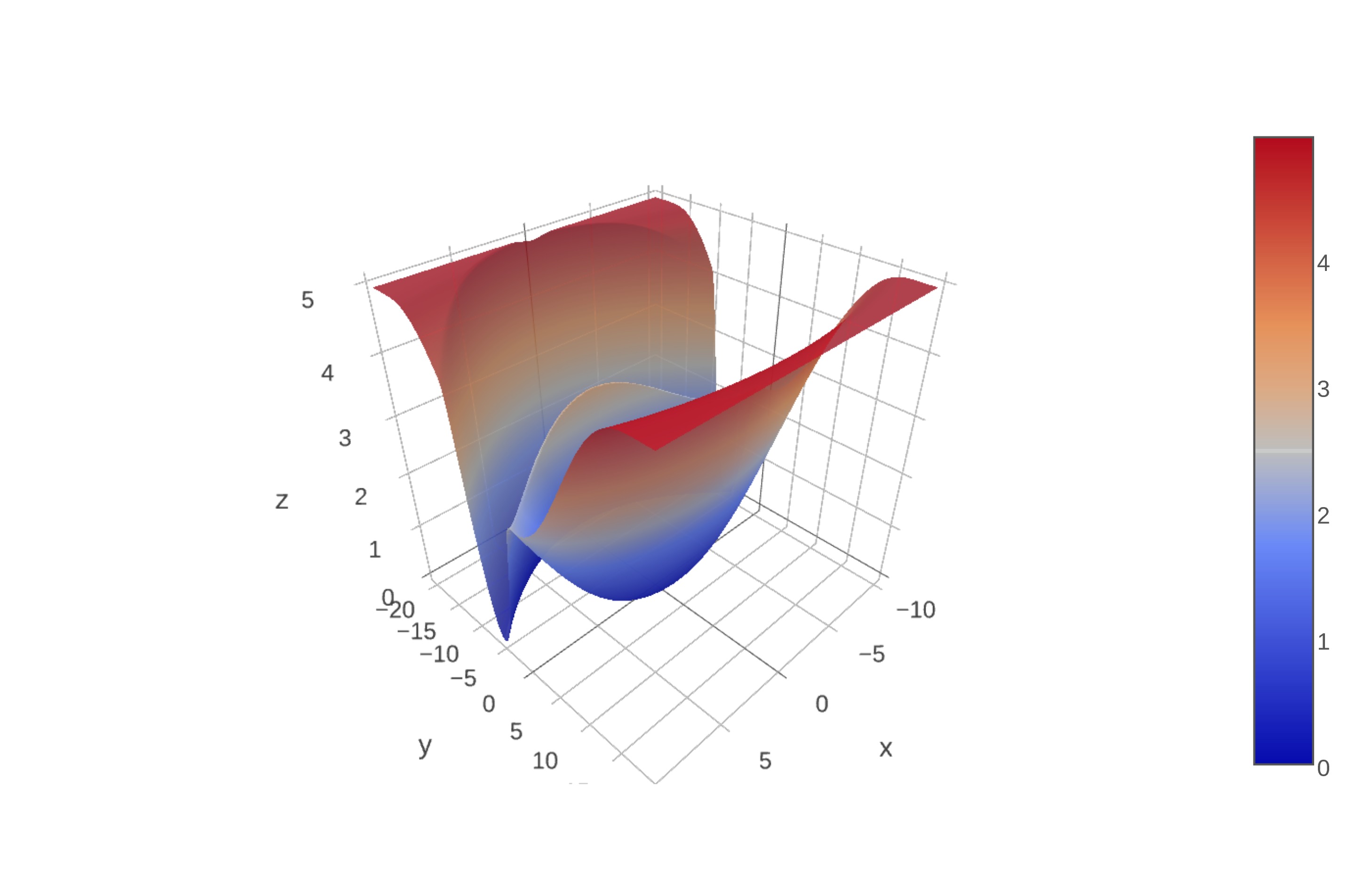}}
\subfloat[Model 2]{\includegraphics[width=0.45\textwidth]{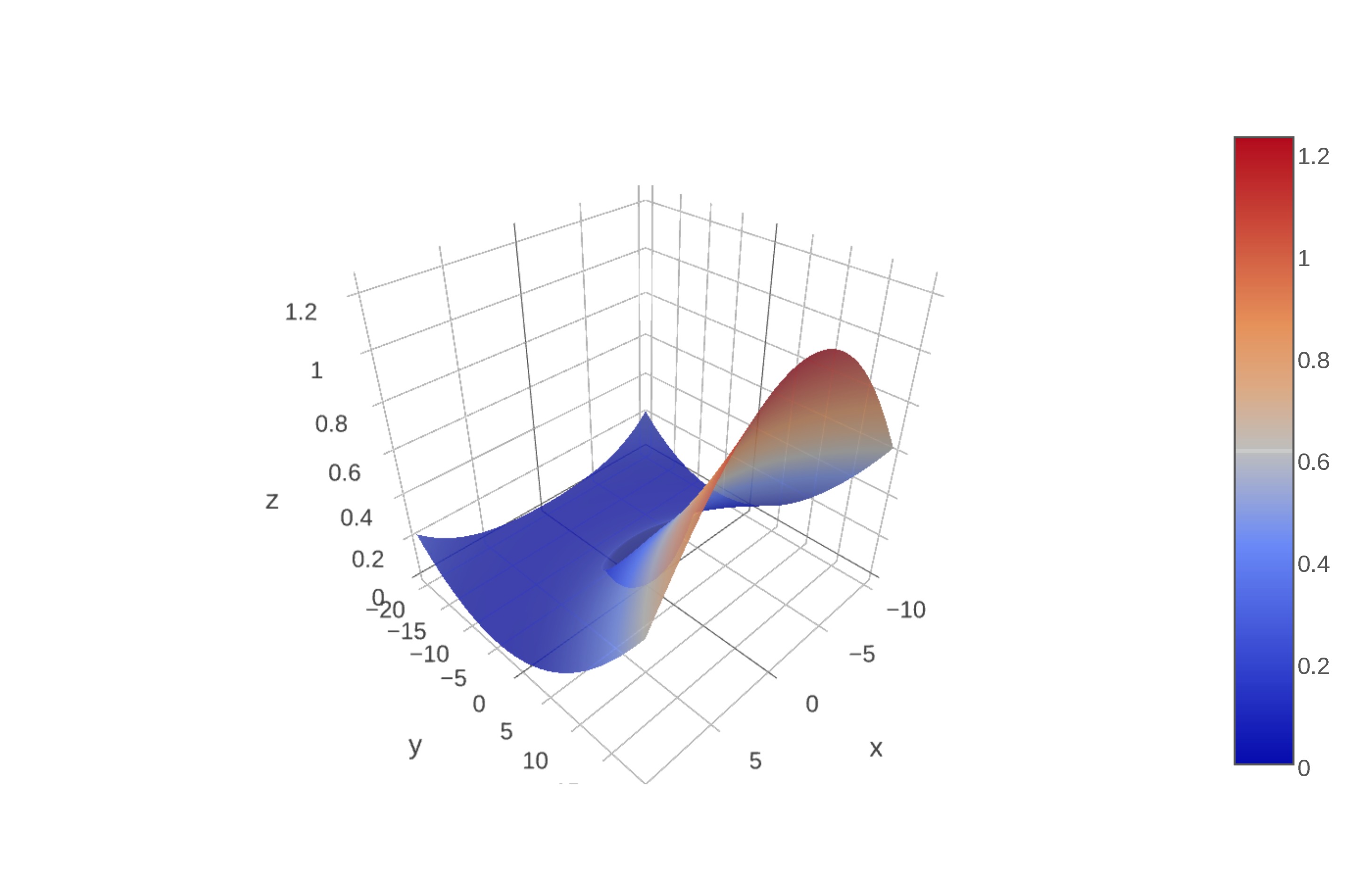}} \\
\subfloat[Model 3]{\includegraphics[width=0.45\textwidth]{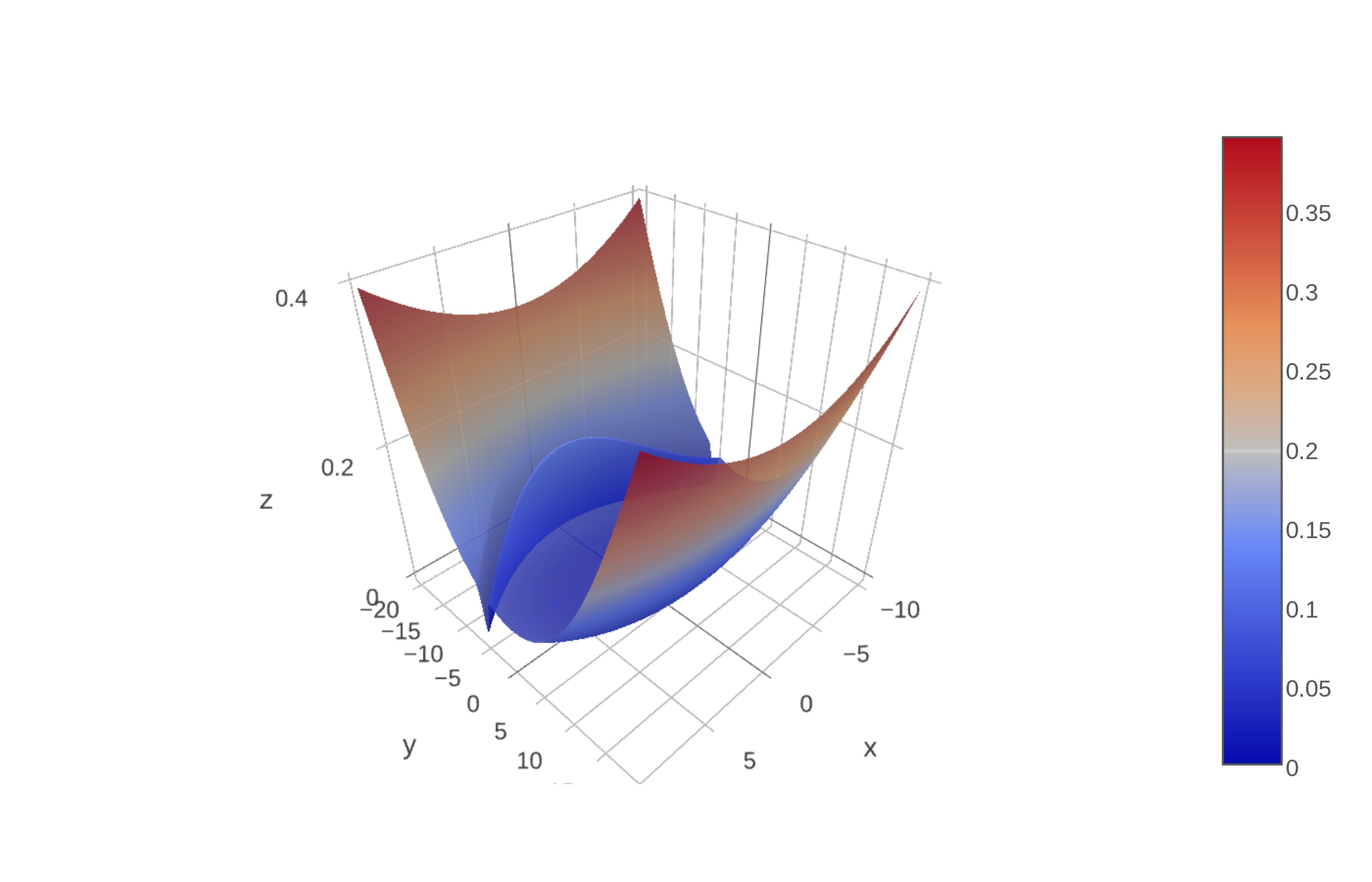}}
\subfloat[Model 4]{\includegraphics[width=0.45\textwidth]{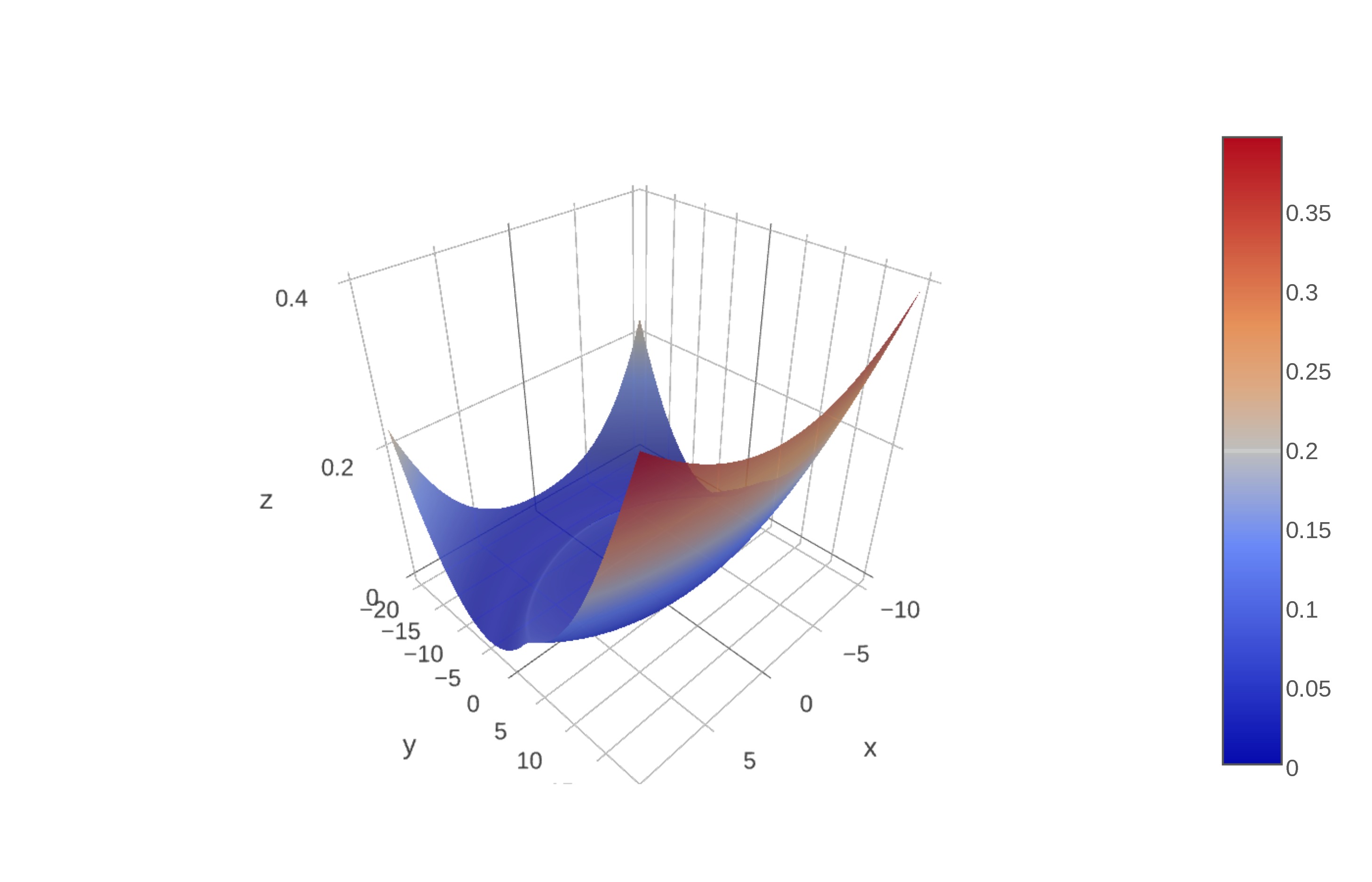}}
\caption[Model Quadratic-Circle Sums Problems]{The expected objective function for the four model quadratic-circle sums objective functions.}
\label{figure-sgd:qc-models}
\end{figure}

To compute the divergence and convergence thresholds for these models, we can precisely evaluate the expected values because the random function that we are considering is a discrete random variable. Moreover, in computing these thresholds, we will use an $\epsilon \approx 2 \times 10^{-2}$ and we compute the integrals using a Monte Carlo sample of size $1000$. The computed thresholds are tabulated in Tables \ref{table:lb} and \ref{table:ub}. Note, for the divergence threshold, $k_{\max}$ is the largest integer smaller than $\frac{s_f}{\lambda_{m-1} \lambda_m}$, and, for the upper bounds, $k_{\max}$ is the the integer closest to $\frac{t_f}{\lambda_m \lambda_1}.$

\begin{table}[ht]
\centering
\tbl{Estimated lower bound values for divergence for the two minimizers of the Quadratic-Circular Sums objective function for the four model problems.}
{
\begin{tabular}{llccccrcc} \toprule
  &  &\multicolumn{4}{c}{Circ. Minimum} &  &\multicolumn{2}{c}{Quad. Minimum} \\
    \cmidrule{3-6} \cmidrule{8-9}
  &  & $k=1$ & $ k = 0.99k_{\max} $ & $ k = 2.0 k_{\max} $ & $ k = \infty $ &  & $ k =  1 $ & $ k = \infty $ \\
    \midrule
Model 1 & & $ 3.560e+11 $ & $ 3.583e+11 $ & $ 3.583e+11 $ & $ 3.583e+11 $ & & $ 4.977e+00 $ & $ 4.981e+00 $ \\
Model 2 & & $ 3.339e+11 $ & $ 3.360e+11 $ & $ 3.360e+11 $ & $ 3.360e+11 $ & & $ 9.954e+02 $ & $ 9.961e+02 $ \\
Model 3 & & $ 1.266e+15 $ & $ 1.273e+15 $ & $ 1.273e+15 $ & $ 1.273e+15 $ & & $ 4.977e+00 $ & $ 4.981e+00 $ \\
Model 4 & & $ 1.269e+15 $ & $ 1.277e+15 $ & $ 1.277e+15 $ & $ 1.277e+15 $ & & $ 9.954e+02 $ & $ 9.961e+02 $ \\
\bottomrule
\end{tabular}
}
\label{table:lb}
\end{table}

\begin{table}[ht]
\centering
\tbl{Estimated upper bound values for convergence for the two minimizers of the Quadratic-Circular Sums objective function for the four model problems.}
{
\begin{tabular}{llccccrcc} \toprule
  &  &\multicolumn{4}{c}{Circ. Minimum} &  &\multicolumn{2}{c}{Quad. Minimum} \\
    \cmidrule{3-6} \cmidrule{8-9}
  &  & $k=1$ & $ k = 0.99k_{\max} $ & $ k = 2.0 k_{\max} $ & $ k = \infty $ &  & $ k =  1 $ & $ k = \infty $ \\
    \midrule
Model 1 & & $ 8.318e-05 $ & $ 1.541e+11 $ & $ 2.013e+11 $ & $ 2.874e+11 $ & & $ 4.977e+00 $ & $ 4.981e+00 $ \\
Model 2 & & $ 8.456e-05 $ & $ 1.624e+11 $ & $ 2.157e+11 $ & $ 3.180e+11 $ & & $ 9.954e+02 $ & $ 9.961e+02 $ \\
Model 3 & & $ 5.655e-06 $ & $ 6.222e+14 $ & $ 8.282e+14 $ & $ 1.226e+15 $ & & $ 4.977e+00 $ & $ 4.981e+00 $ \\
Model 4 & & $ 5.511e-06 $ & $ 6.267e+14 $ & $ 8.264e+14 $ & $ 1.201e+15 $ & & $ 9.954e+02 $ & $ 9.961e+02 $ \\
\bottomrule
\end{tabular}
}
\label{table:ub}
\end{table}

\subsection{Experimental Procedure}  
We study four experimental factors: the model, the learning method, the initialization point, and the learning rate. 
The model factor has four levels as given by the four model objective functions shown in Figure \ref{figure-sgd:qc-models}. The learning method has two levels which are SGD-$1$ or GD. 
The initialization point also has two levels: either the initialization point is selected randomly with a uniform probability from a disc of radius $10^{-8}$ centered about the circular basin minimizer, or the initialization point is selected randomly with a uniform probability from a disc of radius $10^{-8}$ centered about the quadratic basin minimizer. The learning rate has levels that are conditional on the initialization point. 
If the initialization point is near the minimizer of the circular basin, then the learning rate takes on values $10^{10}$, $5(10^{10})$, $10^{11}$, $5(10^{11})$, $10^{12}$, and $5(10^{12})$. 
If the initialization point is near the minimizer of the circular basin, then the learning rate takes on values $1, 4, 16, 64, 256, 1024$. 

For each unique collection of the levels of the four factors, one hundred independent runs are executed with at most twenty iterations. For each run, the euclidean distance between the iterates and the circular basin minimizer and the euclidean distance between the iterates and the quadratic basin minimizer are recorded. 

\subsection{Results and Discussion}
Note that the results of SGD-$k$ on Models 1 and 2 and Models 3 and 4 are similar when initialized around the circular basin minimizer, and the results of SGD-$k$ on Models 1 and 3 and Models 2 and 4 are similar when initialized around the quadratic basin minimizer. Hence, when we discuss the circular basin minimizer, we will compare Models 1 and 3, but we could have just as easily replaced Model 1 with Model 2 or Model 3 with Model 4 and the discussion would be identical. Similarly, when we discuss the quadratic basin minimizer, we will compare Models 1 and 2, but we could have replaced the results of Model 1 with Model 3 or Model 2 with Model 4 and the discussion would be identical.

\begin{figure}[hbt]
\centering
\includegraphics[width=\textwidth]{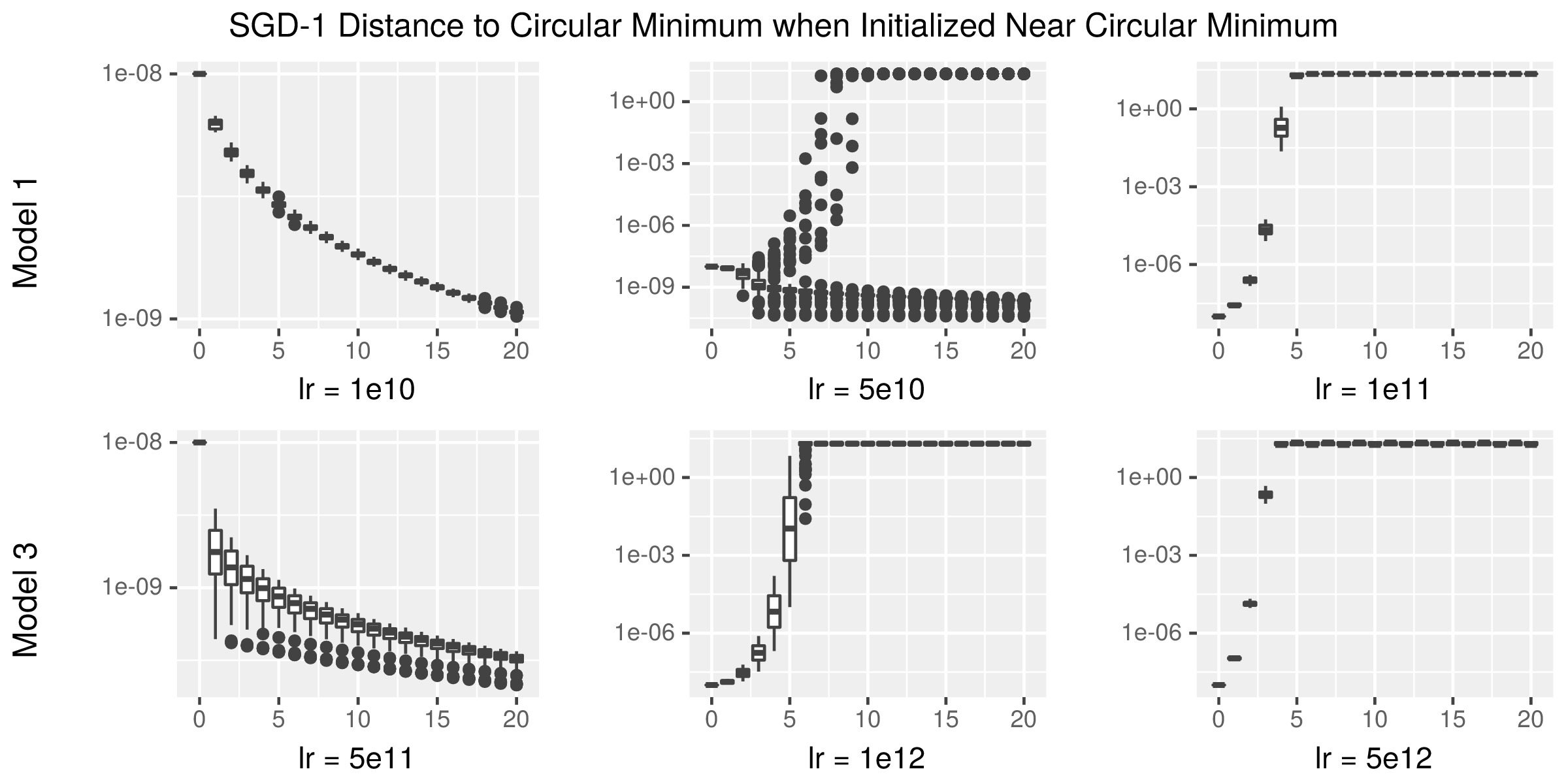}
\caption[SGD-$1$ Near Circle Basin Minimizer]{The behavior of SGD-$1$ on Models 1 and 3 when initialized near the circular minimum. The $y$-axis shows the distance (in logarithmic scale) between the estimates and the circular minimum for all runs of the specified model and the specified learning rate.}
\label{figure-sgd:cq-1}
\end{figure}

\begin{figure}[hbt]
\centering
\includegraphics[width=\textwidth]{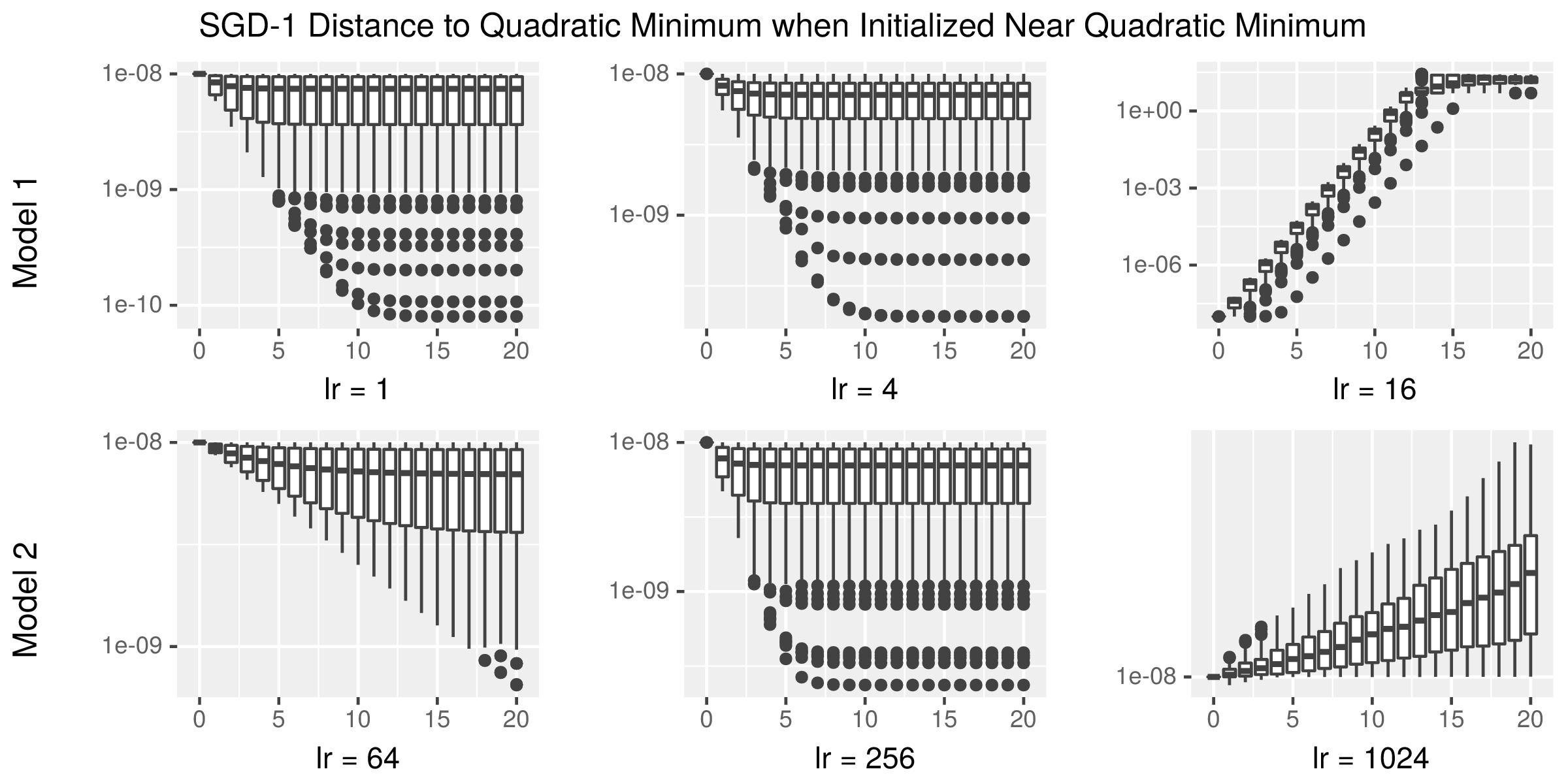}
\caption[SGD-$1$ Near Quadratic Basin Minimizer]{The behavior of SGD-$1$ on Models 1 and 2 when initialized near the quadratic minimum. The $y$-axis shows the distance (in logarithmic scale) between the iterates and the quadratic minimum for all runs of the specified model and the specified learning rate.}
\label{figure-sgd:cq-2}
\end{figure}

Figure \ref{figure-sgd:cq-1} shows the distance between the iterates and the circular basin minimizer for the one hundred independent runs of SGD-$1$ for the specified model and the specified learning rate when initialized near the circular basin minimizer. The learning rates that are displayed are the ones where a transition in the convergence-divergence behavior of the method occur for the specific model. Specifically, SGD-$1$ begins to diverge for learning rates between $5 \times 10^{10}$ and $10^{11}$ for Model 1 and between $5\times 10^{11}$ and  $10^{12}$ for Model 3. 
Similarly, Figure \ref{figure-sgd:cq-2} shows the distance between the iterates and the quadratic basin minimizer for the one hundred independent runs of SGD-$1$ for the specified model and the specified learning rate when initialized near the quadratic basin minimizer. SGD-$1$ begins to diverge for learning rates between $4$ and $16$ for Model 1 and between $256$ and $1024$ for Model 2. 

From these observations, we see that relatively flatter minimizers enjoy larger thresholds for divergence of SGD-$1$ in comparison to sharper minimizers. Moreover, while the bounds computed in Table \ref{table:lb} are conservative (as we would expect), they are still rather informative, especially in the case of the quadratic basin. Thus, regarding questions (1) and (2) above, we see that while our thresholds are slightly conservative, we still correctly predict the expected behavior of the SGD-$k$ iterates. Moreover, regarding question (3), we observe that the iterates diverge from their respective minimizers in a deterministic way with an exponential rate. Again, this observations lends credence to our deterministic mechanism over the widely accepted stochastic mechanism.

\begin{figure}[hbt]
\centering
\includegraphics[width=\textwidth]{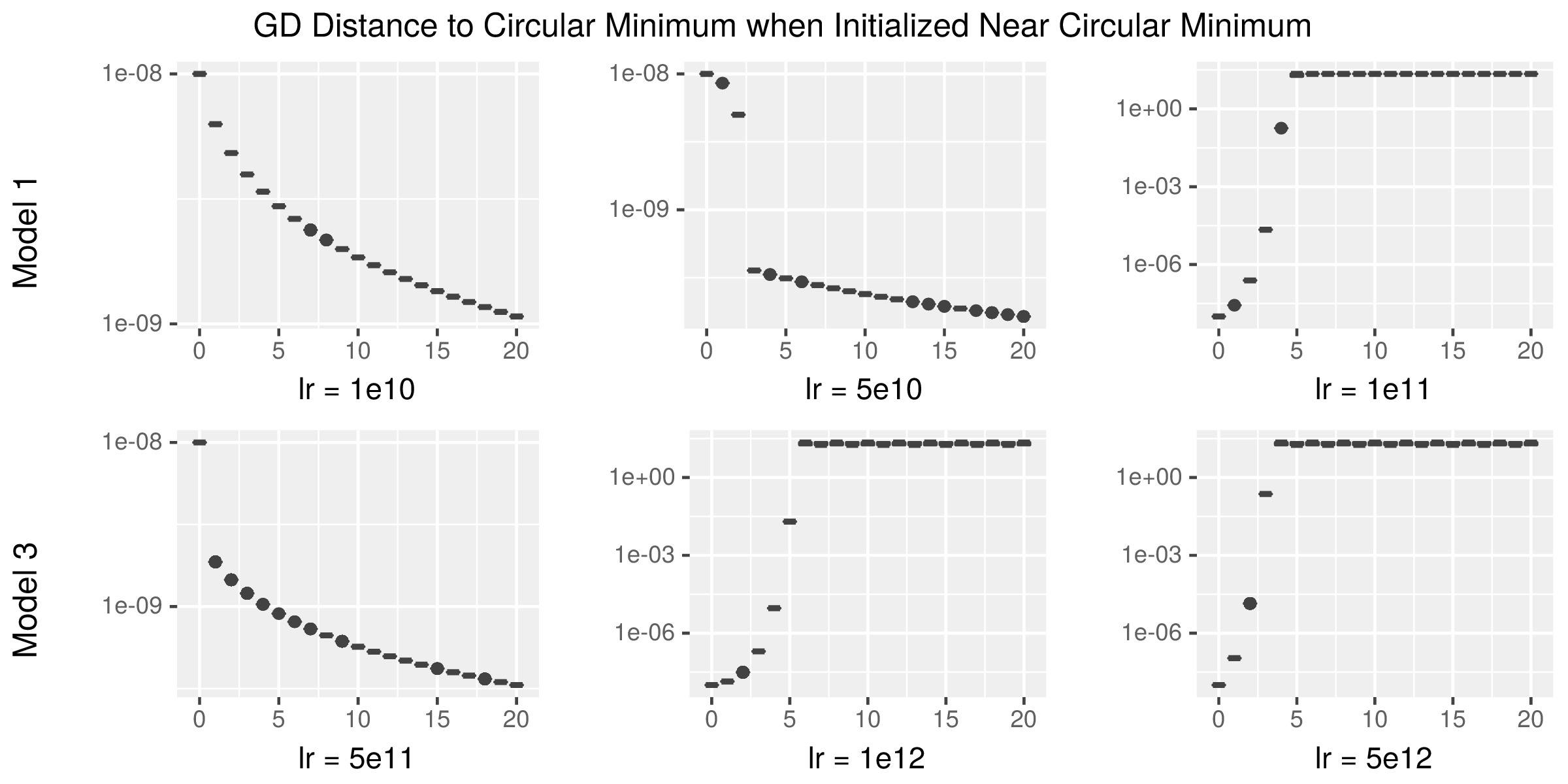}
\caption[GD Near Circle Basin Minimizer]{The behavior of GD on Models 1 and 3 when initialized near the circular basin minimizer. The $y$-axis shows the distance (in logarithmic scale) between the iterates and the circular minimizer for all runs of the specified model and the specified learning rate.}
\label{figure-sgd:cq-3}
\end{figure}

Figure \ref{figure-sgd:cq-3} shows the distance between the iterates and the circular basin minimizer for the one hundred independent runs of gradient descent (GD) for the specified model and the specified learning rate when initialized near the circular basin minimizer. If we compare Figures \ref{figure-sgd:cq-1} and \ref{figure-sgd:cq-3}, we notice that, for Model 1 and learning rate $5 \times 10^{10}$, the runs of GD converge whereas some of the runs for SGD-$1$ diverge. Although we do not report the results for all of the models or all of the learning rates, we note the boundary for divergence-convergence for GD are smaller than those of SGD-$1$. In light of our deterministic mechanism, this behavior is expected: as the batch-size, $k$, increases, the lower bound on the learning rates for divergence increases. Therefore, we should expect that at those boundary learning rates where some SGD-$1$ runs are stable and others diverge, for GD, we should only see stable runs, and, indeed, this is what the comparison of Figures \ref{figure-sgd:cq-1} and \ref{figure-sgd:cq-3} shows.

\begin{figure}[hbt]
\centering
\includegraphics[width=\textwidth]{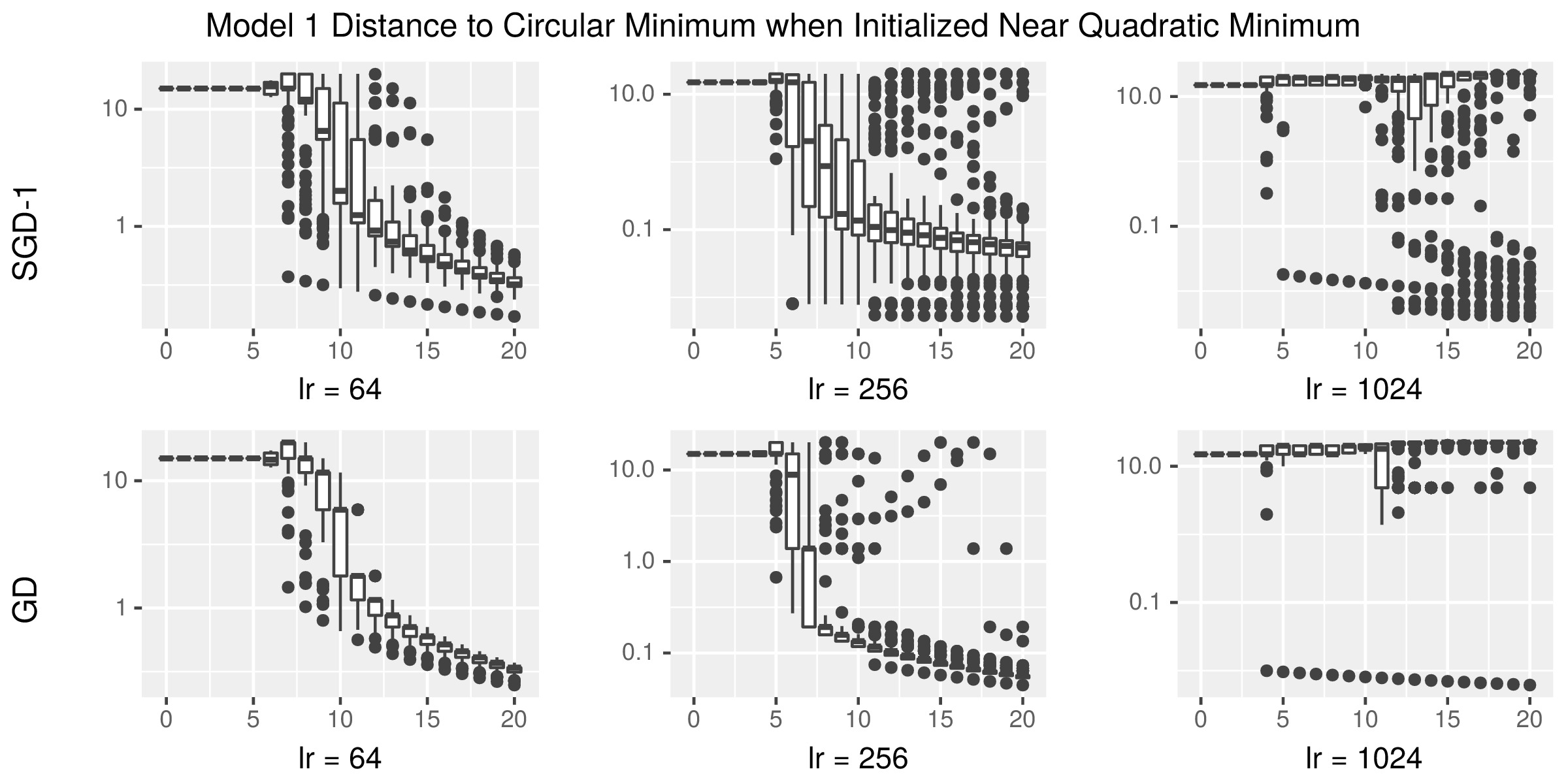}
\caption[Comparing SGD-$1$ and GD Near Quadratic Basin Minimizer]{The behavior of SGD-$1$ and GD on Model 1 when initialized near the quadratic minimum for select learning rates. The $y$-axis shows the distance (in logarithmic scale) between the iterates and the circular minimizer for all runs of the specified method and the specified learning rate.}
\label{figure-sgd:cq-4}
\end{figure}

Figure \ref{figure-sgd:cq-4} shows the distance between the iterates and the circular basin minimizer for the one hundred independent runs of SGD-$1$ and the one hundred independent runs of GD for the specified method and the specified learning rate when initialized near the quadratic basin minimizer. In Figure \ref{figure-sgd:cq-4}, we see that for the learning rates that lead to divergence from the quadratic basin minimizer (compare to the top three subplots of Figure \ref{figure-sgd:cq-1}) are in some cases able to converge to the circular minimizer. Again, in light of our deterministic mechanism, this is expected: the learning rates shown in Figure \ref{figure-sgd:cq-4} guarantee divergence from the quadratic minimizer and are sufficiently small that they can lead to convergence to the circular minimizer. However, we notice in Figure \ref{figure-sgd:cq-4} that as the learning rate increases, even though the learning rate is below the divergence bound for the circular minimizer, the iterates for SGD-$1$ and GD are diverging from both the circular and quadratic minima and are converging to the corners of the feasible region.

To summarize, we see that convergence-divergence behavior of SGD-$k$ for the quadratic-circle sums nonconvex problem is captured by our deterministic mechanism. Although our estimated bounds are sometimes conservative, they generally divide the convergence and divergence regions of SGD-$k$. Moreover, as our deterministic mechanism predicts, we observed exponential divergence away from minimizers when the learning rate is above the divergence threshold. Again, this lends credence to our mechanism for divergence and provides evidence against the stochastic mechanism for SGD-$k$ iterates to ``escape'' sharp minimizers.

\section{Numerical Study of an Inhomogeneous Nonconvex Problem} \label{section:inhomogeneous}
In order to further elaborate our results, we now consider a less trivial nonconvex optimization problem which has multiple \textit{inhomogeneous} minimizers, a more complex surface, and is much higher dimensional. Specifically, we explore the modification of Styblinksi-Tang (ST) function \citep{styblinski1990}, which was originally introduced in the neural network literature and served as a test function owing to its complex surface yet simple mathematical form. The $p$-dimensional ST function is given by
\begin{equation} \label{eqn-sgd:st}
f(x) = \frac{1}{2}\sum_{i=1}^p c_{i,1} x_i^4 + c_{i,2} x_i^2 + c_{i,3}x_i,
\end{equation}
where $c_{i,1},c_{i,3}$ are non-negative scalars, and $c_{i,2}$ is a nonpositive scalar. Using this equation we can define the following stochastic nonconvex problem.

\begin{problem}[Styblinski-Tang Sums] \label{problem-sgd:st-sums}
Let $f_1,\ldots,f_N$ be functions as specified by (\ref{eqn-sgd:st}) of some dimension $p$. The Styblinksi-Tang sums objective function is
\begin{equation} \label{eqn-sgd-problem:sts-obj}
\sum_{i=1}^N p_i f_i(x),
\end{equation}
where $p_i$ are positive valued and sum to one. Let $Y$ be a random variable taking values $\lbrace f_i: i =1,\ldots,N \rbrace$ with the probability of sampling $f_i$ equals to $p_i$. Let $Y_1,Y_2,\ldots$ be independent copies of $Y$. The Styblinski-Tang Sums problem is to use $\lbrace Y_1,Y_2,\ldots \rbrace$ to minimize (\ref{eqn-sgd-problem:sts-obj}) restricted to $(-5,5)^p \subset \mathbb{R}^p$.
\end{problem}

Just as above, we will use this problem to either verify or negate our deterministic mechanism. We will focus on the following three questions.
\begin{enumerate}
\item Is our predicted threshold for divergence valid?
\item If we analogously compute a threshold for stability, will it be valid?
\item In the case of divergence, do we observe exponential divergence?
\end{enumerate}
\begin{remark}
Our approach to studying SGD-$k$ on Problem \ref{problem-sgd:st-sums} will be similar to our study of SGD-$k$ on Problem \ref{problem-sgd:qcs} in the sense that we will use constant step sizes. Therefore, the use of constant step sizes will prevent SGD-$k$ from converging to the minimizer, but we should expect stability around the minimizer. Importantly, by using constant step sizes, we remove the added variable of how to schedule the learning rate appropriately.
\end{remark}

We consider three randomly generated realizations of the problem, which we label Model 1, Model 2, and Model 3. Model 1 is of dimension ten, has $2^{10}$ minimizers, and has $N=200$ components; Model 2 is of dimension fifty, has $2^{50}$ minimizes, and has $N=1000$ components; Model 3 is of dimension one hundred, has $2^{100}$ minimizers, and has $N = 2000$ components. As we cannot study all of these minimizers, we will study the sharpest minimizer and the flattest minimizer of each model. For these minimizers, the divergence and convergence thresholds are computed in a manner analogous to the computations done in \S \ref{section:homogeneous}, and are reported in Tables \ref{table-sgd:st-lb} and \ref{table-sgd:st-ub}.

\begin{table}[ht]
\centering
\tbl{Estimated lower bound values for divergence for the sharpest and flattest minima of different ST Sums problems.}
{
\begin{tabular}{lccccccccc} \toprule
 &  &\multicolumn{2}{c}{Model 1} & &\multicolumn{2}{c}{Model 2} & &\multicolumn{2}{c}{Model 3} \\ 
 \cmidrule{3-4}  \cmidrule{6-7}  \cmidrule{9-10} 
 k &  & Flat & Sharp &  & Flat & Sharp &  & Flat & Sharp \\ 
 \midrule 
$ 1.0 $ &  & $ 9.001e-03 $ & $ 7.682e-03 $ & & $ 2.134e-02 $ & $ 1.951e-02 $ & & $ 2.598e-02 $ & $ 1.871e-02 $ \\ 
 $ 100.0 $ &  & $ 9.814e-02 $ & $ 8.133e-02 $ & & $ 6.831e-01 $ & $ 6.117e-01 $ & & $ 1.221e+00 $ & $ 8.561e-01 $ \\ 
 $ 200.0 $ &  & $ 1.033e-01 $ & $ 8.546e-02 $ & & $ 8.100e-01 $ & $ 7.225e-01 $ & & $ 1.590e+00 $ & $ 1.106e+00 $ \\ 
 $ 350.0 $ &  & $ 1.057e-01 $ & $ 8.737e-02 $ & & $ 8.800e-01 $ & $ 7.833e-01 $ & & $ 1.827e+00 $ & $ 1.264e+00 $ \\ 
 $ 500.0 $ &  & $ 1.067e-01 $ & $ 8.815e-02 $ & & $ 9.116e-01 $ & $ 8.106e-01 $ & & $ 1.942e+00 $ & $ 1.341e+00 $ \\ 
 $ \infty $ &  & $ 1.091e-01 $ & $ 9.004e-02 $ & & $ 9.947e-01 $ & $ 8.823e-01 $ & & $ 2.279e+00 $ & $ 1.563e+00 $ \\ 
 \bottomrule
\end{tabular}
}
\label{table-sgd:st-lb}
\end{table}

\begin{table}[ht]
\centering
\tbl{Estimated upper bound values for convergence for the sharpest and flattest minima of different ST sums problems.}
{
\begin{tabular}{lccccccccc} \toprule
 &  &\multicolumn{2}{c}{Model 1} & &\multicolumn{2}{c}{Model 2} & &\multicolumn{2}{c}{Model 3} \\ 
 \cmidrule{3-4}  \cmidrule{6-7}  \cmidrule{9-10} 
 k &  & Flat & Sharp &  & Flat & Sharp &  & Flat & Sharp \\ 
 \midrule 
$ 1.0 $ &  & $ 2.319e-03 $ & $ 2.250e-03 $ & & $ 8.603e-04 $ & $ 8.334e-04 $ & & $ 7.587e-04 $ & $ 9.089e-04 $ \\ 
 $ 100.0 $ &  & $ 4.268e-02 $ & $ 3.781e-02 $ & & $ 7.925e-02 $ & $ 7.621e-02 $ & & $ 7.345e-02 $ & $ 8.594e-02 $ \\ 
 $ 200.0 $ &  & $ 4.632e-02 $ & $ 4.087e-02 $ & & $ 1.174e-01 $ & $ 1.049e-01 $ & & $ 1.423e-01 $ & $ 1.629e-01 $ \\ 
 $ 350.0 $ &  & $ 4.808e-02 $ & $ 4.234e-02 $ & & $ 1.329e-01 $ & $ 1.180e-01 $ & & $ 2.379e-01 $ & $ 2.303e-01 $ \\ 
 $ 500.0 $ &  & $ 4.882e-02 $ & $ 4.296e-02 $ & & $ 1.403e-01 $ & $ 1.243e-01 $ & & $ 2.812e-01 $ & $ 2.539e-01 $ \\ 
 $ \infty $ &  & $ 5.064e-02 $ & $ 4.447e-02 $ & & $ 1.613e-01 $ & $ 1.418e-01 $ & & $ 3.742e-01 $ & $ 3.337e-01 $ \\ 
 \bottomrule
\end{tabular}
}
\label{table-sgd:st-ub}
\end{table}

\subsection{Experimental Procedure}
We study four experimental factors: the model, the batch size, the initialization point, and the learning rate. The model factor has three levels that correspond to the three models described above. The batch size will take the values corresponding to SGD-$1$, SGD-$200$, SGD-$500$, and Gradient Descent. The initialization point will be randomly selected from a uniform distribution on a ball whose radius will take values $\lbrace 10^{-3},10^{-2},10^{-1}, 1 \rbrace$ and is centered at either the sharpest minimizer or the flattest minimizer. The learning rates will be linear combinations of the upper and lower bounds. Specifically, if $u$ is the threshold for convergence and $l$ is the threshold for divergence, then the learning rates will be $1.5l$, $0.5(u+l)$, or $0.5u$. For each unique collection of the levels of the four factors, one hundred independent runs are executed with at most twenty iterations. For each run, the euclidean distances between the iterates and the minimizer near the initialization point are recorded. 

\subsection{Results and Discussion}
Note that the results Models 1, 2 and 3 are nearly identical for the purposes of our discussion, and so we will feature Model 1 only in our discussion below.

\begin{figure}[hbtp]
\centering 
\includegraphics[width=\textwidth]{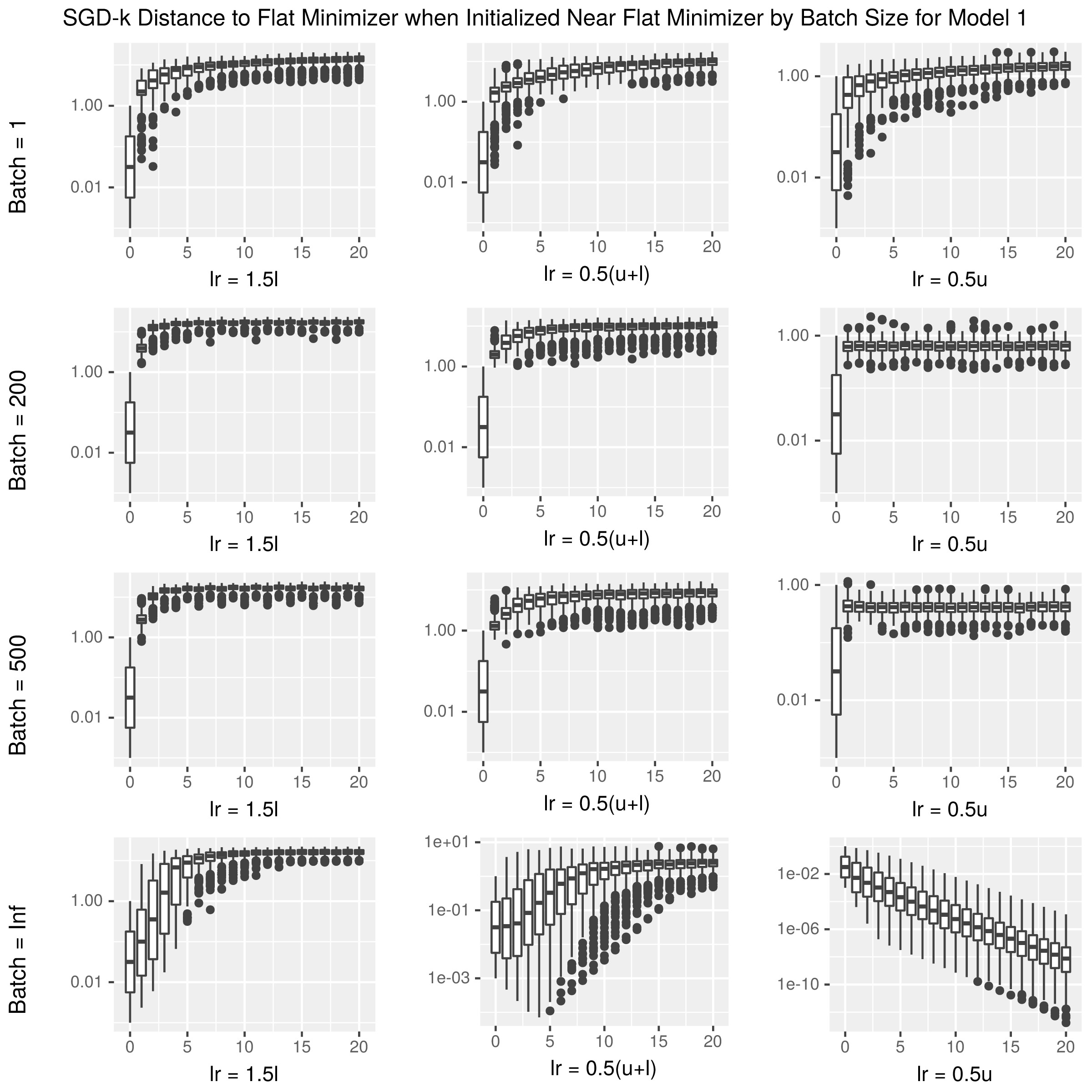}
\caption[SGD-$k$ on Styblinski-Tang Near Flat Minimizer]{The behavior of SGD-$k$ on Model 1 for different batch sizes when initialized near the flat minimizer. The $y$-axis shows the distance (in logarithmic scale) between the iterates and the flat minimizer for all runs of the specified batch size and specified learning rate.}
\label{figure-sgd:st-1}
\end{figure}

\begin{figure}[htbp]
\centering 
\includegraphics[width=\textwidth]{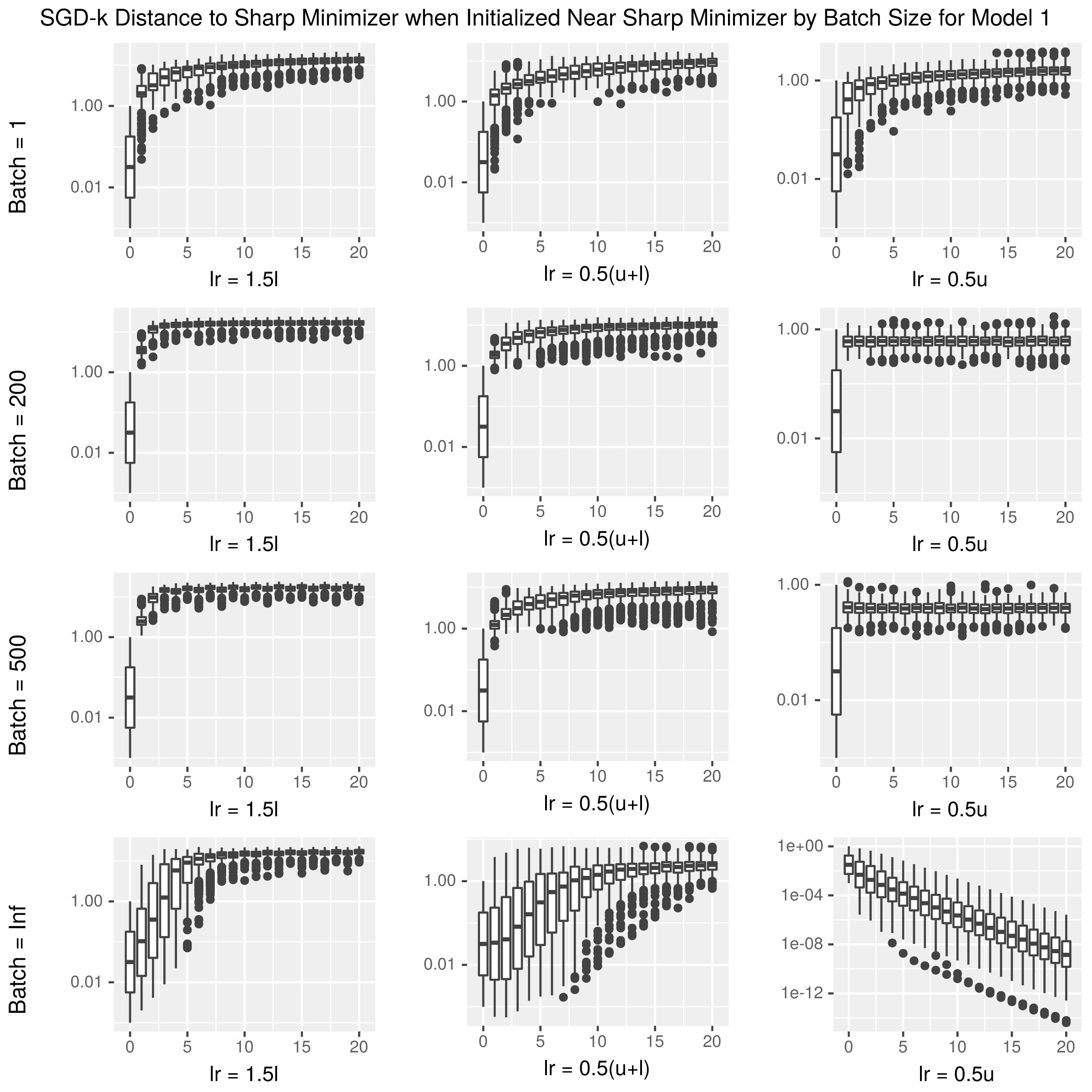}
\caption[SGD-$k$ on Styblinski-Tang Near Sharp Minimizer]{The behavior of SGD-$k$ on Model 1 for different batch sizes when initialized near the sharp minimizer. The $y$-axis shows the distance (in logarithmic scale) between the iterates and the sharp minimizer for all runs of the specified batch size and specified learning rate.}
\label{figure-sgd:st-3}
\end{figure}

\begin{figure}[htbp]
\centering
\includegraphics[width=\textwidth]{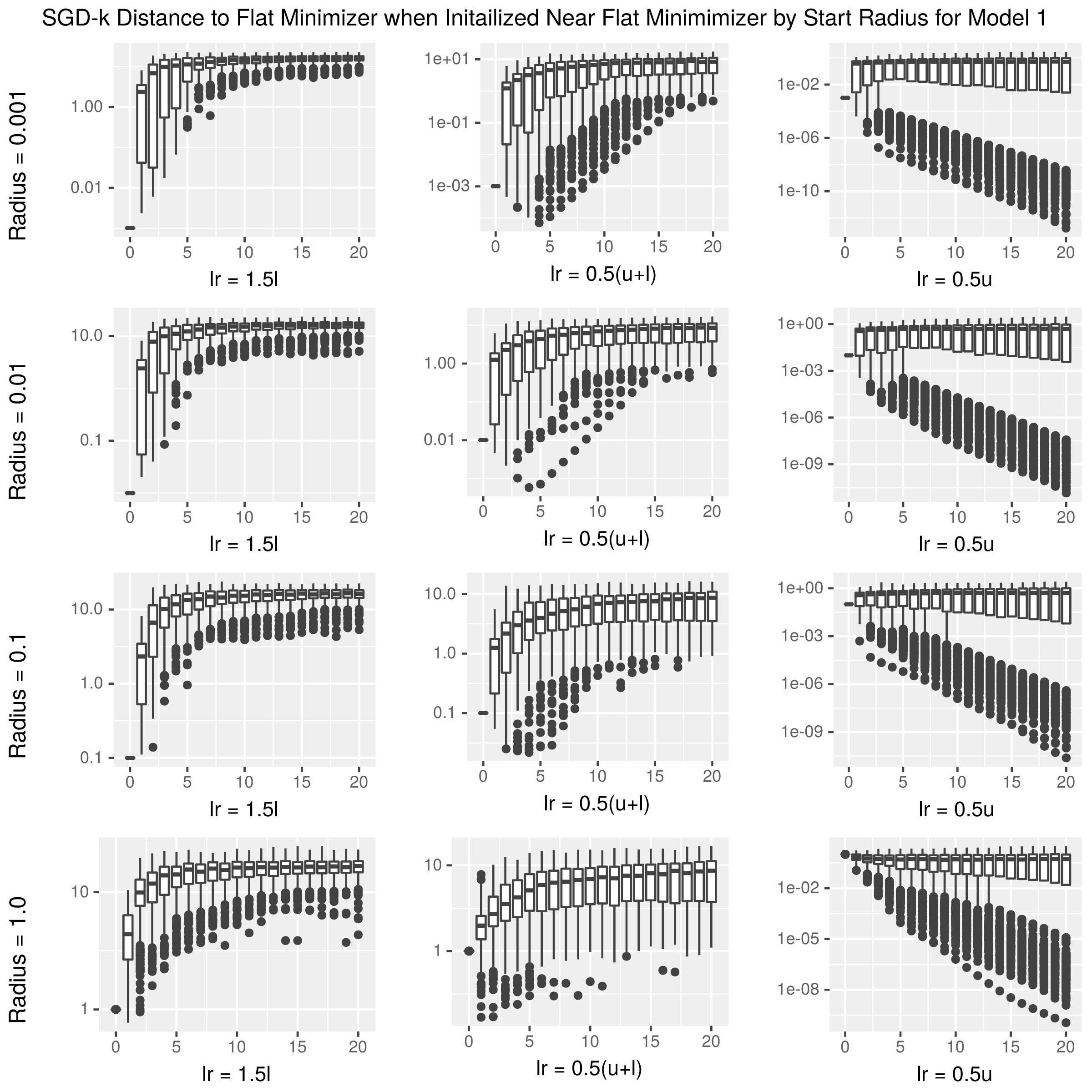}
\caption[Initialization and SGD-$k$ on Styblinski-Tang Near Flat Minimizer]{The behavior of SGD-$k$ on Model 1 for different starting radii when initialized near the flat minimizer. The $y$-axis shows the distance (in logarithmic scale) between the iterates and the flat minimizer for all runs of the specified starting radius and specified learning rate.}
\label{figure-sgd:st-2}
\end{figure}

\begin{figure}[hbtp]
\centering 
\includegraphics[width=\textwidth]{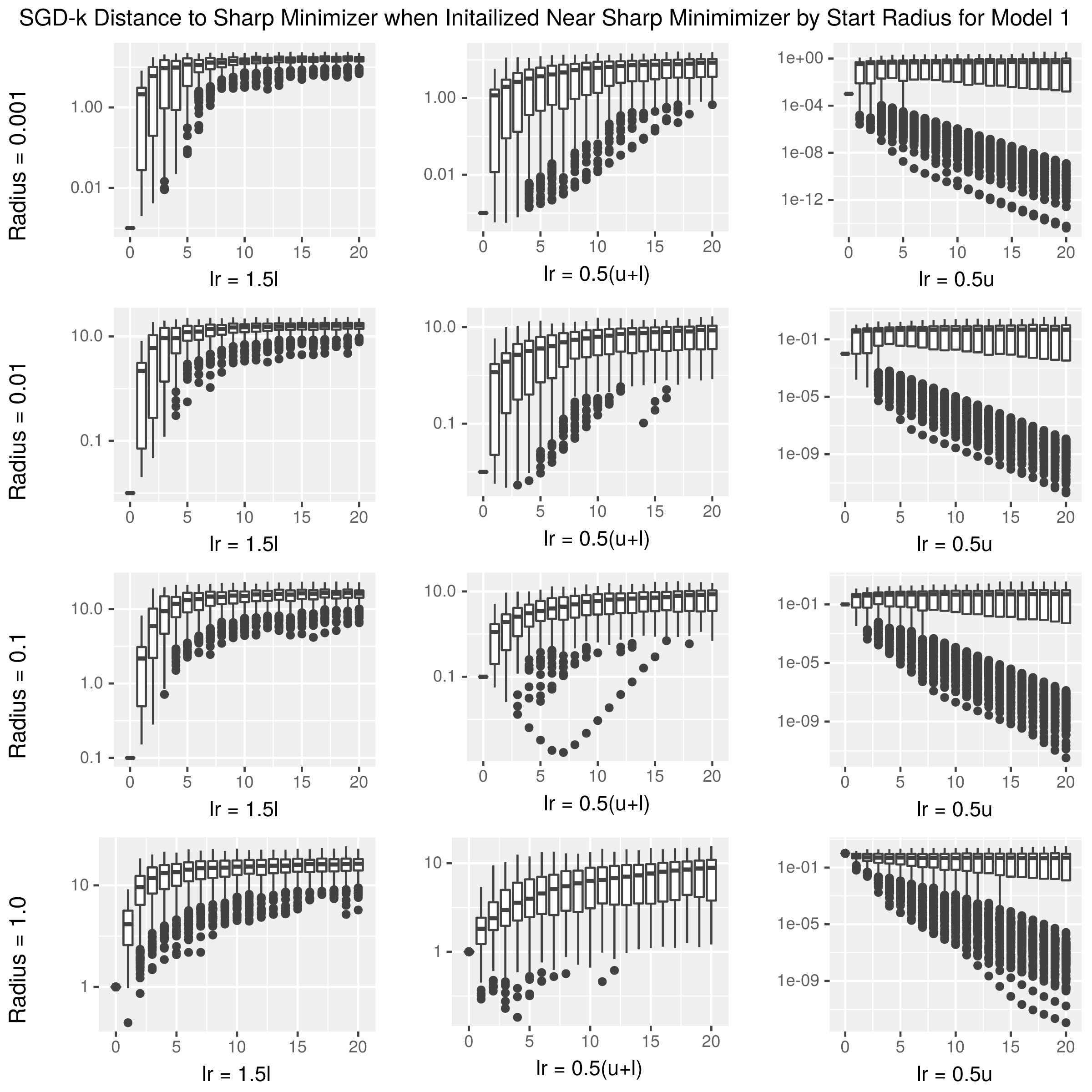}
\caption[Initialization and SGD-$k$ on Styblinski-Tang Near Sharp Minimizer]{The behavior of SGD-$k$ on Model 1 for different starting radii when initialized near the sharp minimizer. The $y$-axis shows the distance (in logarithmic scale) between the iterates and the sharp minimizer for all runs of the specified starting radius and specified learning rate.}
\label{figure-sgd:st-4}
\end{figure}

Figure \ref{figure-sgd:st-1} shows the distance between SGD-$k$ iterates and the flat minimizer on Model 1 for different batch sizes and different learning rates when SGD-$k$ is initialized near the flat minimizer. 
We note that, regardless of batch size, the iterates are diverging from the flat minimizer and are converging to the corners of the feasible region for learning rates $1.5l$ and $0.5(u+l)$. 
On the other hand, for the learning rate $0.5u$, we see stability for $k=1,200,500$ about the minimizer and we see convergence for GD (i.e., $k=\infty$). 
Similarly, Figure \ref{figure-sgd:st-3} shows the distance between SGD-$k$ iterates and the sharp minimizer on Model 1 for different batch sizes and different learning rates when SGD-$k$ is initialized near the sharp minimizer. 
We note that, regardless of batch size, the iterates are diverging from the sharp minimizer and converging to the corners of the feasible region for learning rates $1.5l$ and $0.5(u+l)$. On the other hand, for the learning rate $0.5u$, we see stability for $k=1,200,500$ about the sharp minimizer and we see convergence for GD (i.e., $k = \infty$). 

Taking the results in Figures \ref{figure-sgd:st-1} and \ref{figure-sgd:st-3} together, we see that we are able to use our deterministic mechanism to find learning rates that either ensure divergence from or stability about a minimizer if we know its local geometric properties. Consequently, we again have evidence that our deterministic mechanism can correctly predict the behavior of SGD-$k$ for nonconvex problems. Moreover, when divergence does occur, Figures \ref{figure-sgd:st-1} and \ref{figure-sgd:st-3} also display exponential divergence as predicted by our deterministic mechanism.

For a different perspective, Figure \ref{figure-sgd:st-2} shows the distance between SGD-$k$ iterates and the flat minimizer on Model 1 for different starting radii and different learning rates when SGD-$k$ is initialized near the flat minimizer. We note that, regardless of the starting radius, the iterate are diverging from the flat minimizer and are converging to the corners of the feasible region for learning rates $1.5l$ and $0.5(u+l)$. On the other hand, for the learning rate $0.5u$, we see stability and even convergence for all of the runs regardless of the starting radius. 
Similarly, Figure \ref{figure-sgd:st-4} shows the distance between SGD-$k$ iterates and the sharp minimizer on Model 1 for different starting radii and different learning rates when SGD-$k$ is initialized near the sharp minimizer. We note that, regardless of the starting radius, the iterates are diverging from the sharp minimizer and are converging to the corners of the feasible region for learning rates $1.5l$ and $0.5(u+l)$. On the other hand, for learning rate $0.5u$, we see stability about and even convergence to the sharp minimizer. 

Taking the results in Figures \ref{figure-sgd:st-2} and \ref{figure-sgd:st-4} together, we see that we are able to use our deterministic mechanism to find learning rates that either ensure divergence from, or stability about, a minimizer if we know its local geometric properties.  Consequently, we again have evidence that our deterministic mechanism can correctly predict the behavior of SGD-$k$ for nonconvex problems. Moreover, when divergence does occur, Figures \ref{figure-sgd:st-1} and \ref{figure-sgd:st-3} also display exponential divergence as predicted by our deterministic mechanism.

\section{Conclusion} \label{section:conclusion}
How SGD-$k$ converges to, or diverges from, particular minimizers with distinct geometric properties has been an important question in the machine learning literature. We discussed the widely-accepted stochastic mechanism for this observed phenomenon, and we pointed out some intuitive challenges with this mechanism. By analyzing a generic quadratic problem, we proposed a deterministic mechanism for how SGD escapes from different minimizers, which is based on the batch size of SGD-$k$ and the local geometric properties of the minimizer. We then verified the predictions of our deterministic mechanism on two nontrivial nonconvex problems through experimentation. In particular, we verified the prediction of exponential divergence from minimizers for certain learning rates, which is perhaps the strongest evidence supporting our deterministic mechanism over the stochastic mechanism. Our ongoing work focuses on rigorously extending our divergence and convergence analysis to a local analysis for nonconvex problems in order to provide a precise characterization of these phenomenon in the nonconvex setting. 

\section*{Acknowledgements}
We would like to thank Rob Webber for pointing out the simple proof item 2 in Lemma \ref{lemma-sgd:higher-moment-bounds}. We would also like to thank Mihai Anitescu for his general guidance throughout the preparation of this work.

\section*{Funding}
The author is supported by the NSF Research and Training Grant \# 1547396.

\bibliographystyle{tfs}
\bibliography{bibliography}

\appendix
\section{Quadratic-Circle Problem} \label{section:qc}
Each quadratic-circle function is the minimum of two functions: a quadratic basin function and a circular basin function. The quadratic basin function is defined as
\begin{equation} \label{eqn-sgd:quad-basin}
g(x) = q_1 ( x_2 - q_2 x_1^2 - q_3)^2.
\end{equation}
The quadratic basin function, $g$, traces out a parabola in the plane of its arguments where $g(x)$ is zero and increases quadratically as $x_2$ deviates from this parabola. The parameters $q_1$ and $q_2$ are nonnegative. The parameter $q_1$ determines how quickly $g(x)$ increases off of the parabola (i.e., its sharpness), and the remaining parameters determine the shape of the parabola. Examples of the quadratic basin function are shown in Figure \ref{figure-sgd:quadratic-basin-example}.

\begin{figure}[htbp]
\centering
\subfloat[]{\includegraphics[width=0.45\textwidth]{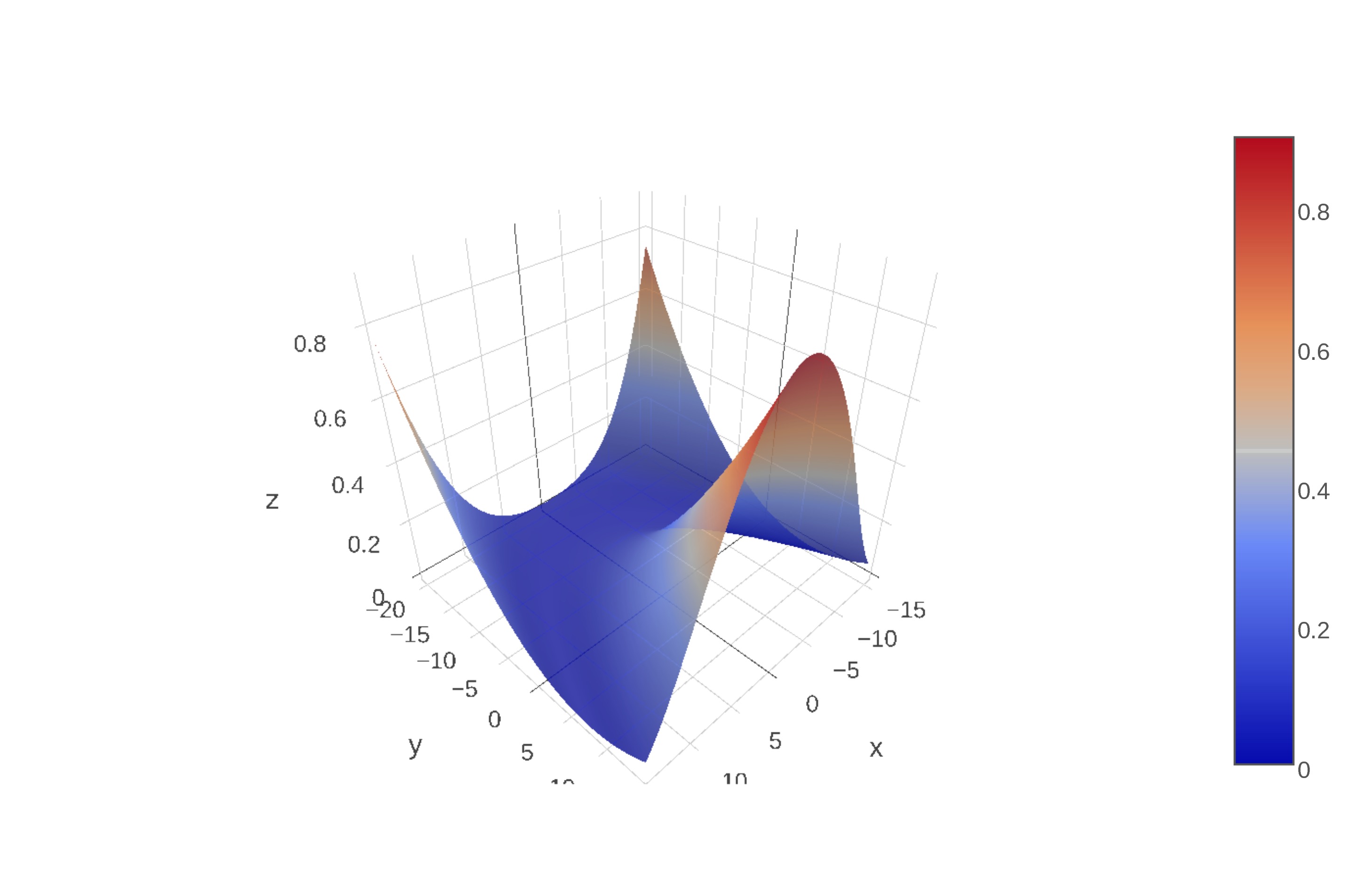}}
\subfloat[]{\includegraphics[width=0.45\textwidth]{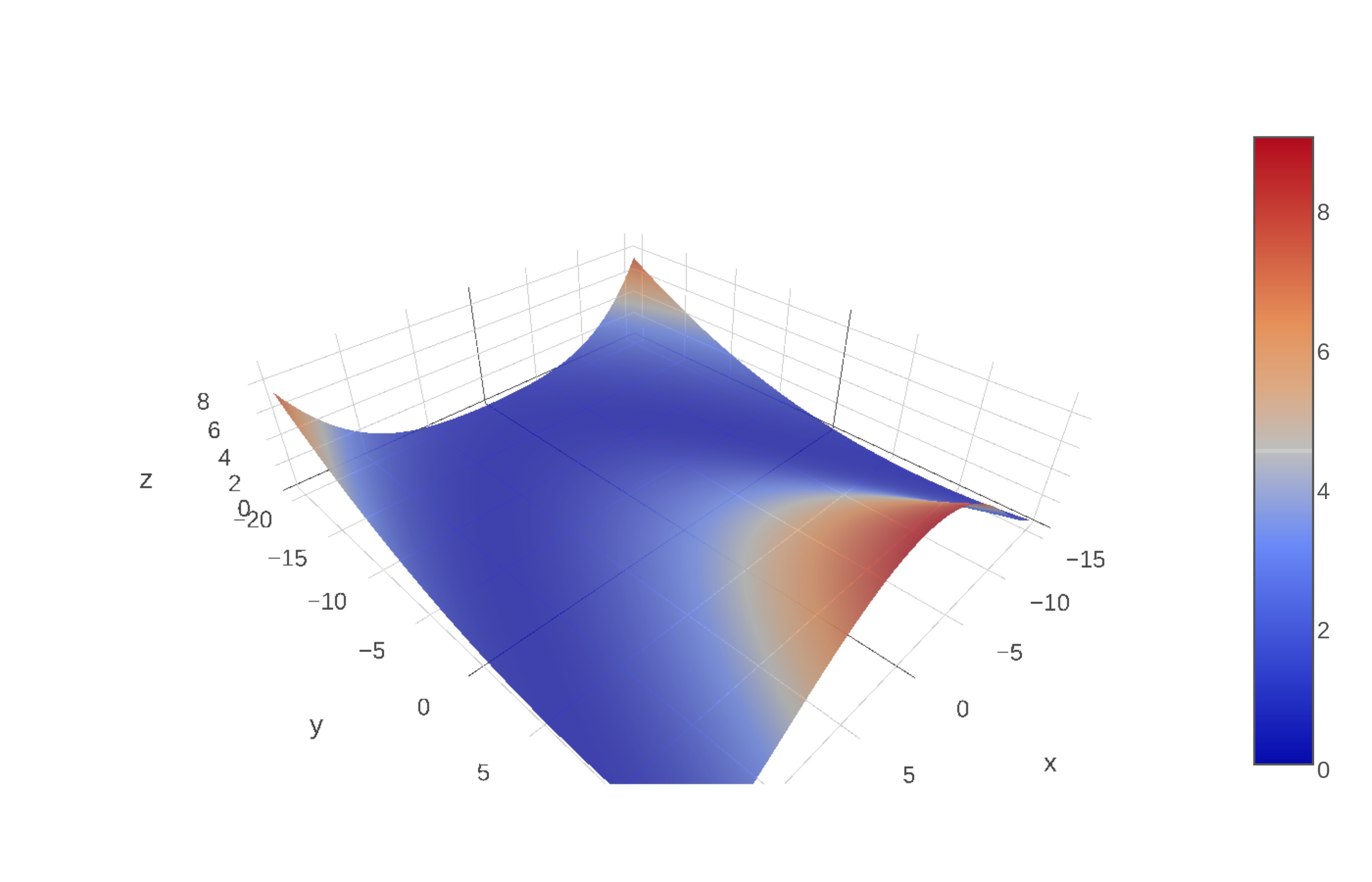}} \\
\subfloat[]{\includegraphics[width=0.45\textwidth]{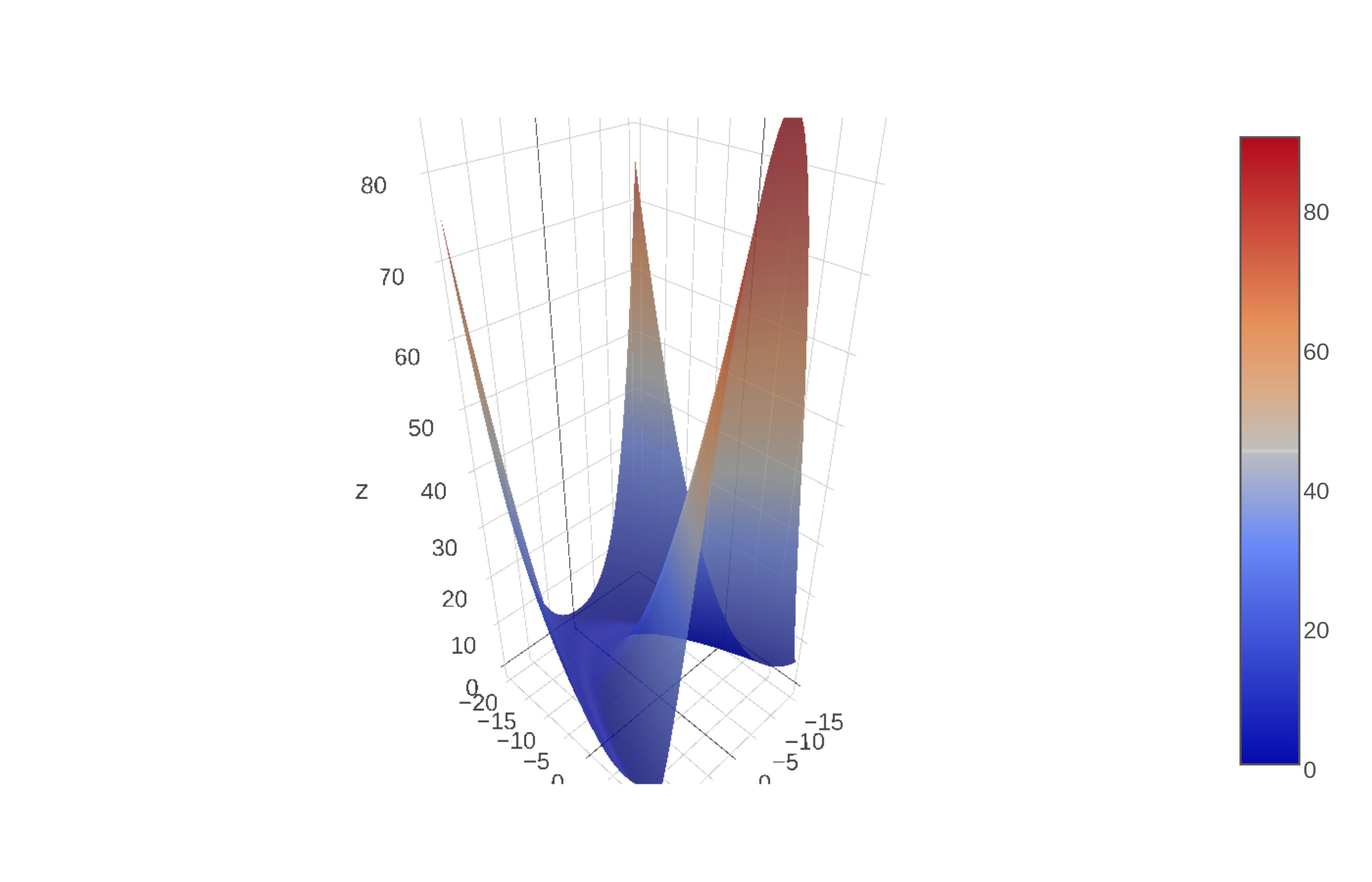}}
\subfloat[]{\includegraphics[width=0.45\textwidth]{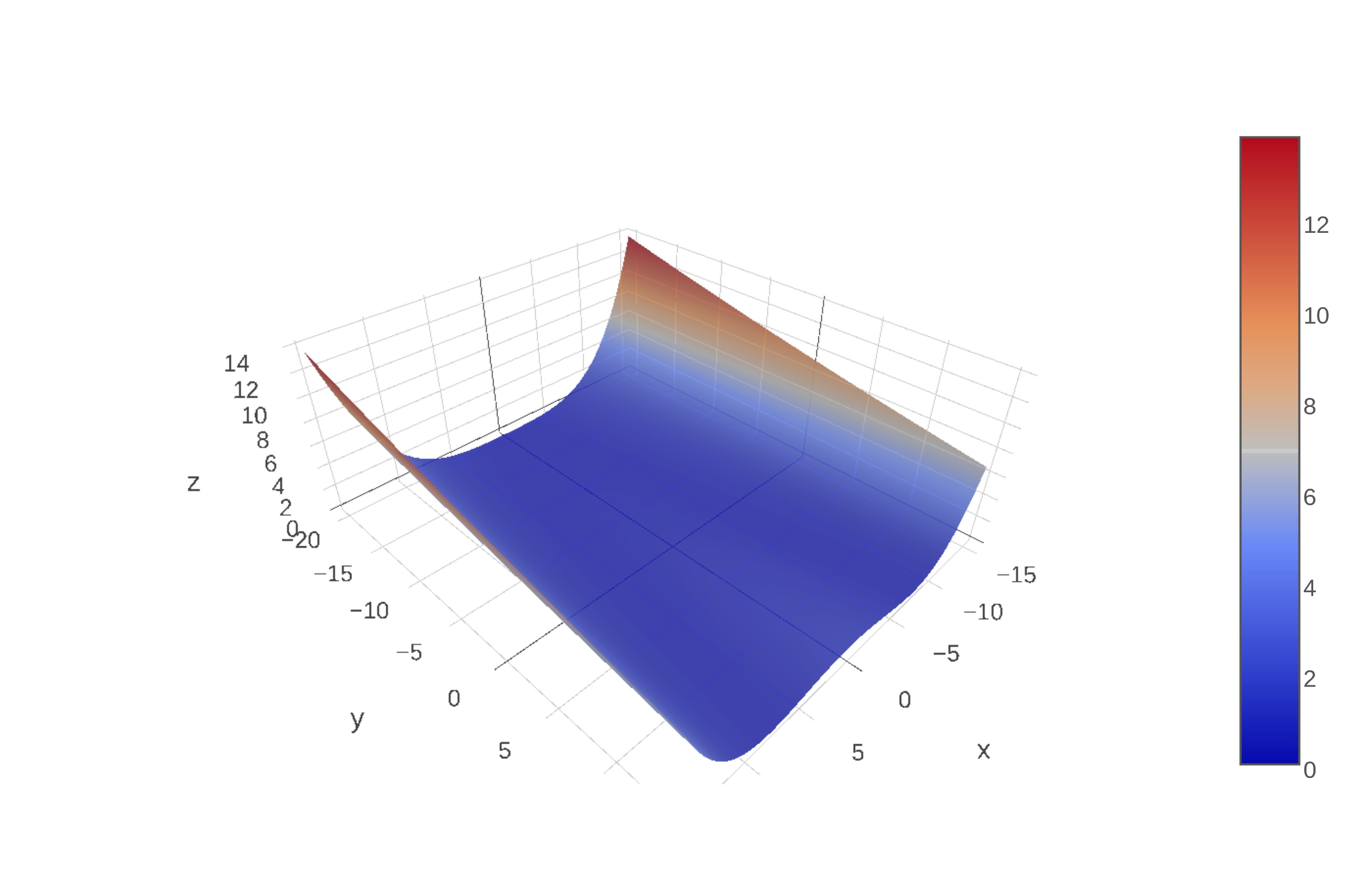}} \\
\subfloat[]{\includegraphics[width=0.45\textwidth]{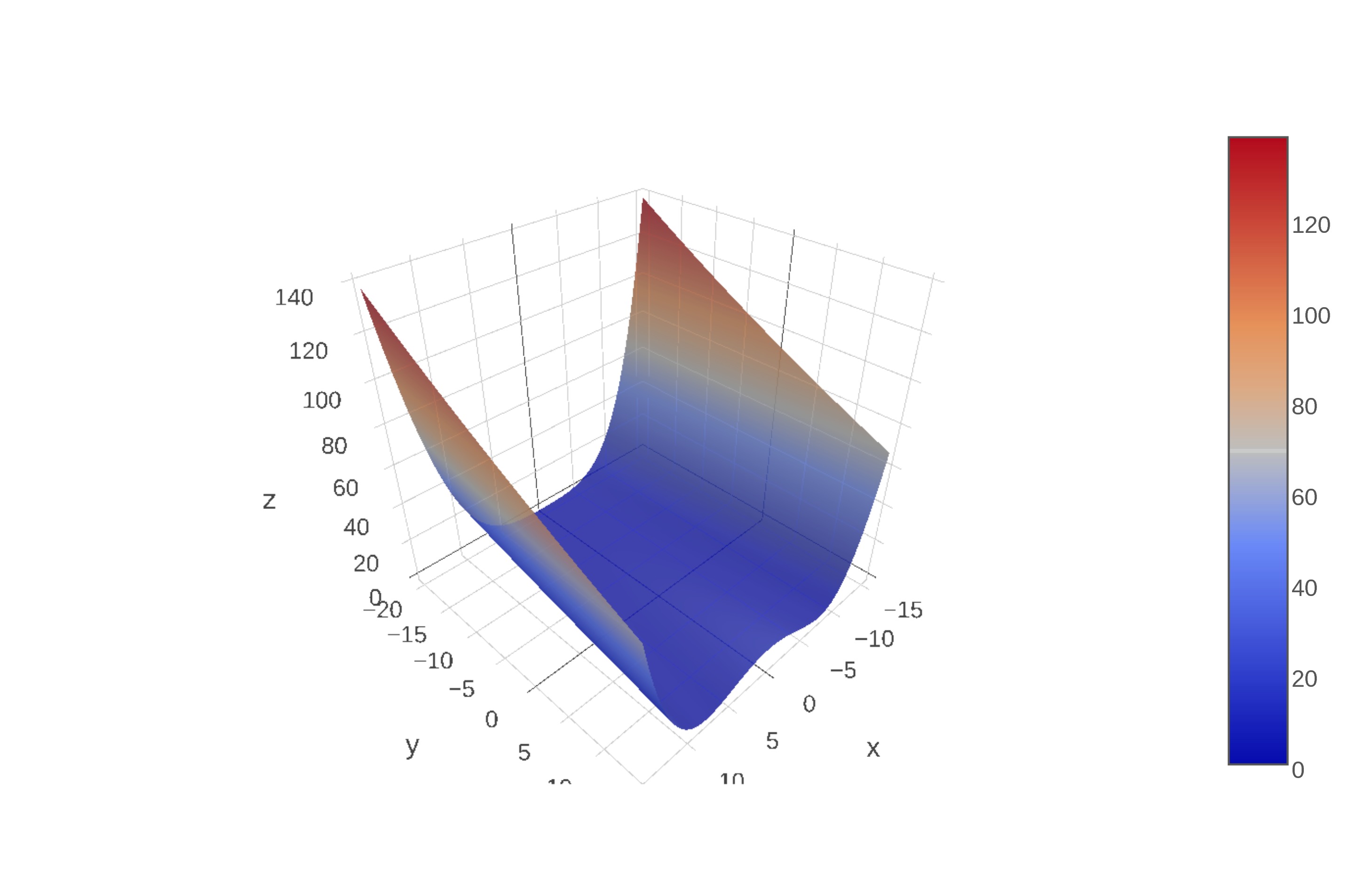}}
\subfloat[]{\includegraphics[width=0.45\textwidth]{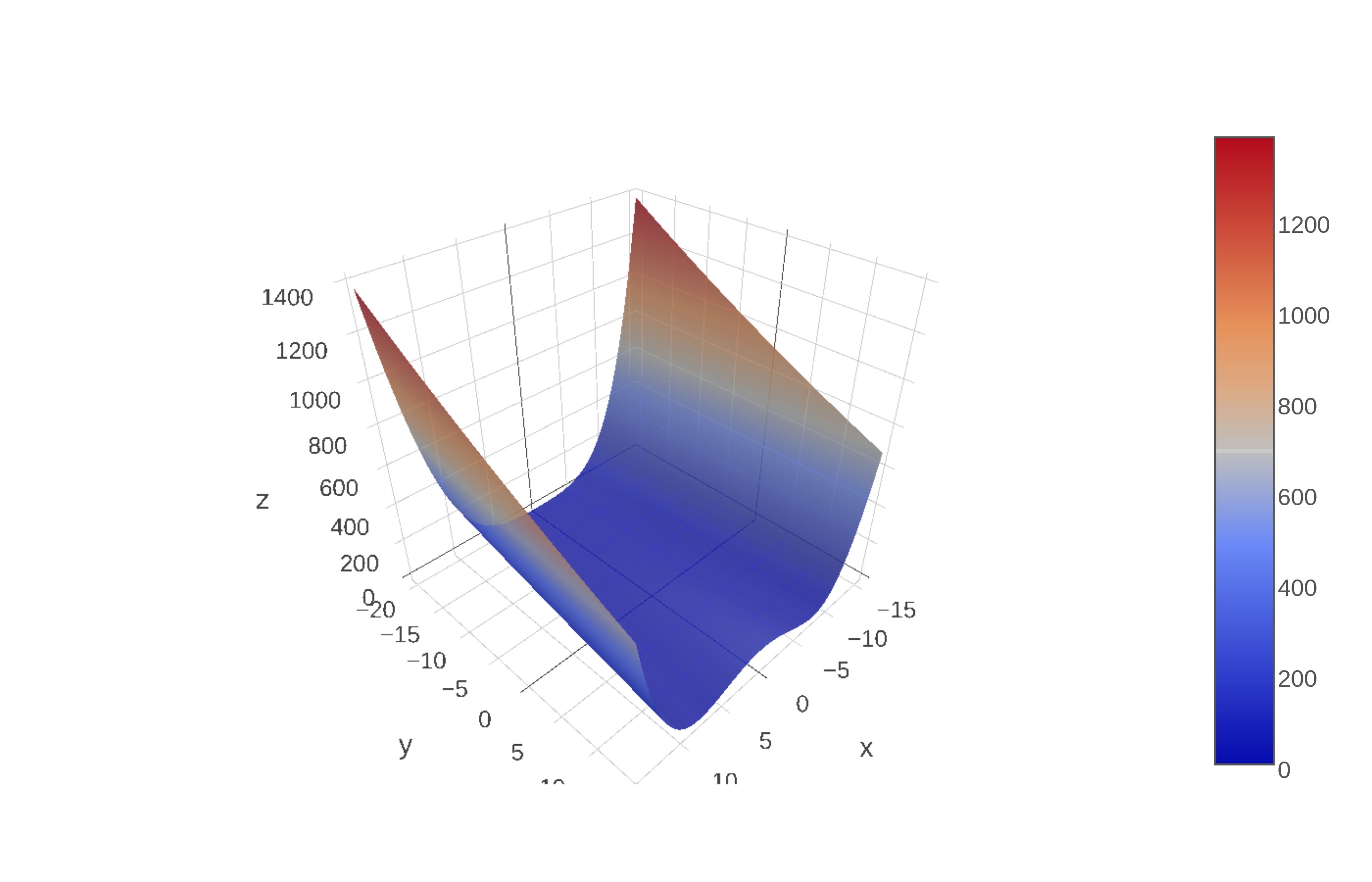}}
\caption[Quadratic Basin Function Examples]{Examples of the quadratic basin functions used in constructing the Quadratic-Circle problem.}
\label{figure-sgd:quadratic-basin-example}
\end{figure}

The circular basin function is a piecewise function defined as
\begin{equation} \label{eqn-sgd:circ-basin}
h(x) = \begin{cases}
c_4 & \norm{x} \leq c_2 \\
c_4+c_1 & \norm{x} \geq c_3 \\
c_4 + c_1 \left( \frac{\norm{x} - c_2}{c_3 - c_2}  \right)^3 \left[6 \left( \frac{\norm{x} - c_2}{c_3 - c_2}  \right)^2 - 15 \left( \frac{\norm{x} - c_2}{c_3 - c_2}  \right) + 10 \right] & \text{ otherwise}
\end{cases},
\end{equation}
where all of the parameters are nonnegative. The circular basin function, $h$, defines a flat region encompassed by a circle of radius $c_2$ and defines a flat region outside of a circle of radius $c_3$. Between these two circles and for $c_2 > 0$, we have a twice differentiable function which ensures that $h(x)$ is twice continuously differentiable. When the ratio between $c_3$ and $c_2$ decreases, the steepness of the basin's walls increases (i.e., the walls are sharper). Increasing the distance between $c_4$ and $c_1$ has a similar outcome. Examples of the circular basin function are shown in Figur \ref{figure-sgd:circular-basin-example}.

\begin{figure}[htbp]
\centering
\subfloat[]{\includegraphics[width=0.45\textwidth]{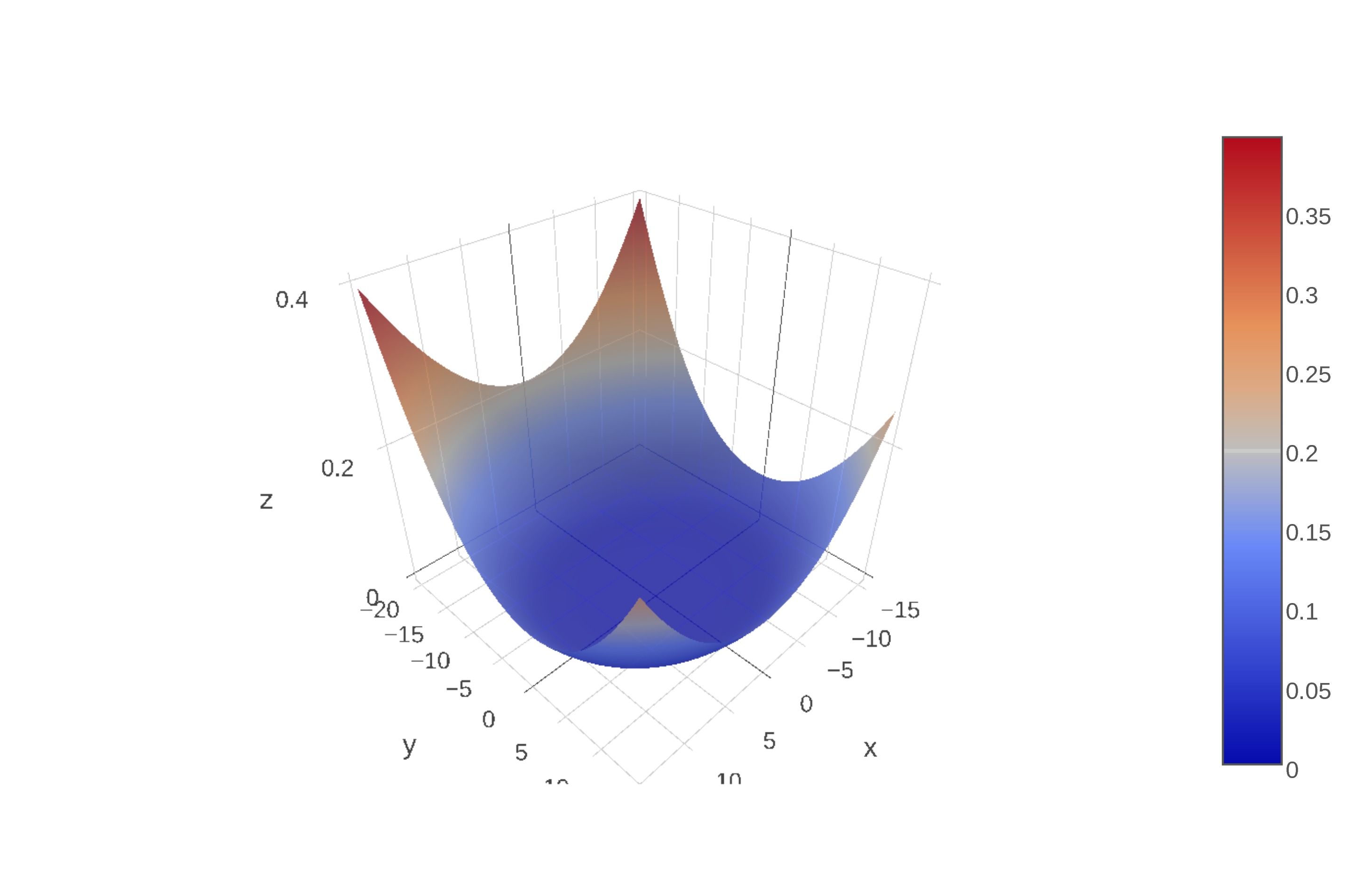}}
\subfloat[]{\includegraphics[width=0.45\textwidth]{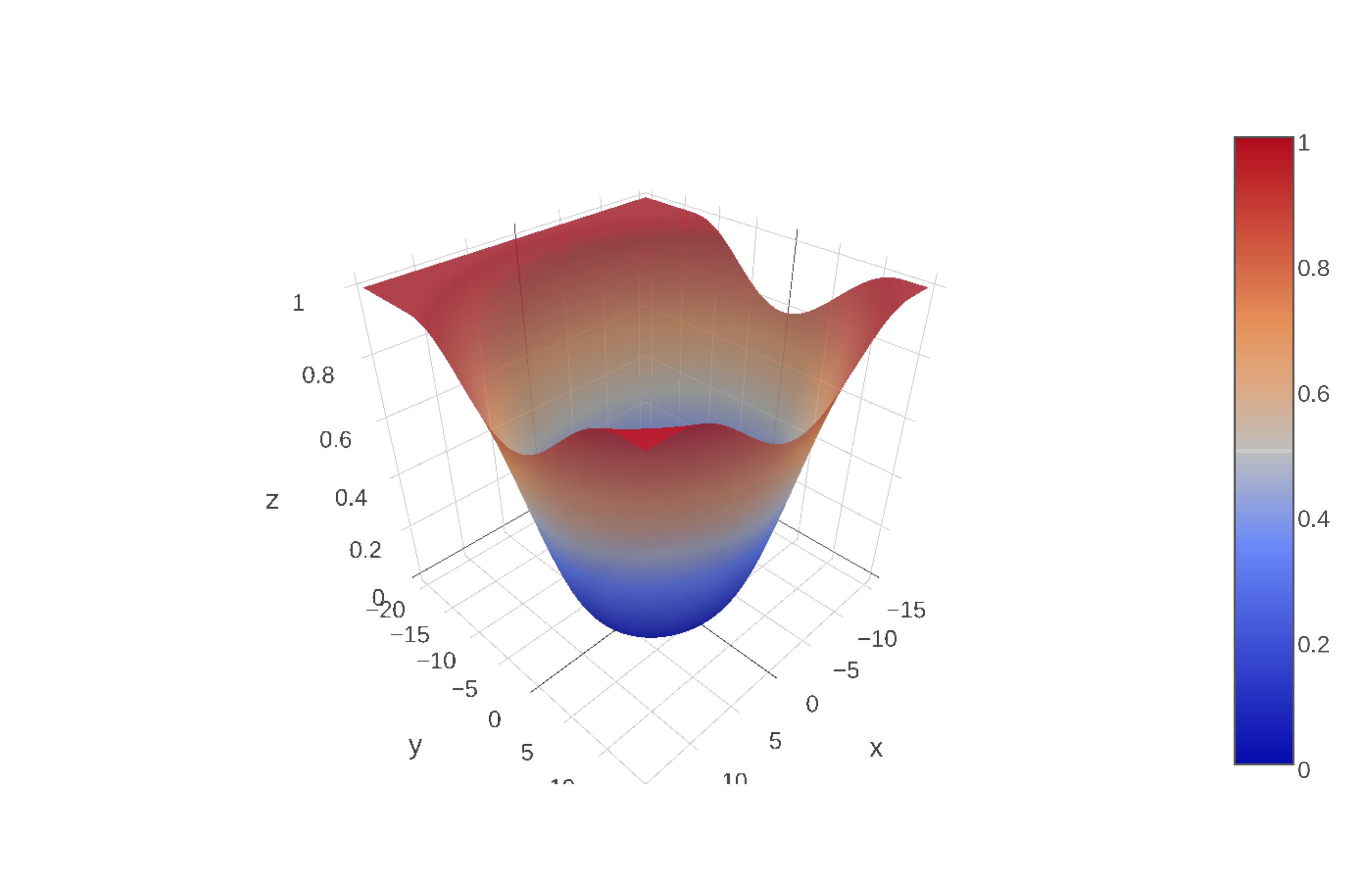}} \\
\subfloat[]{\includegraphics[width=0.45\textwidth]{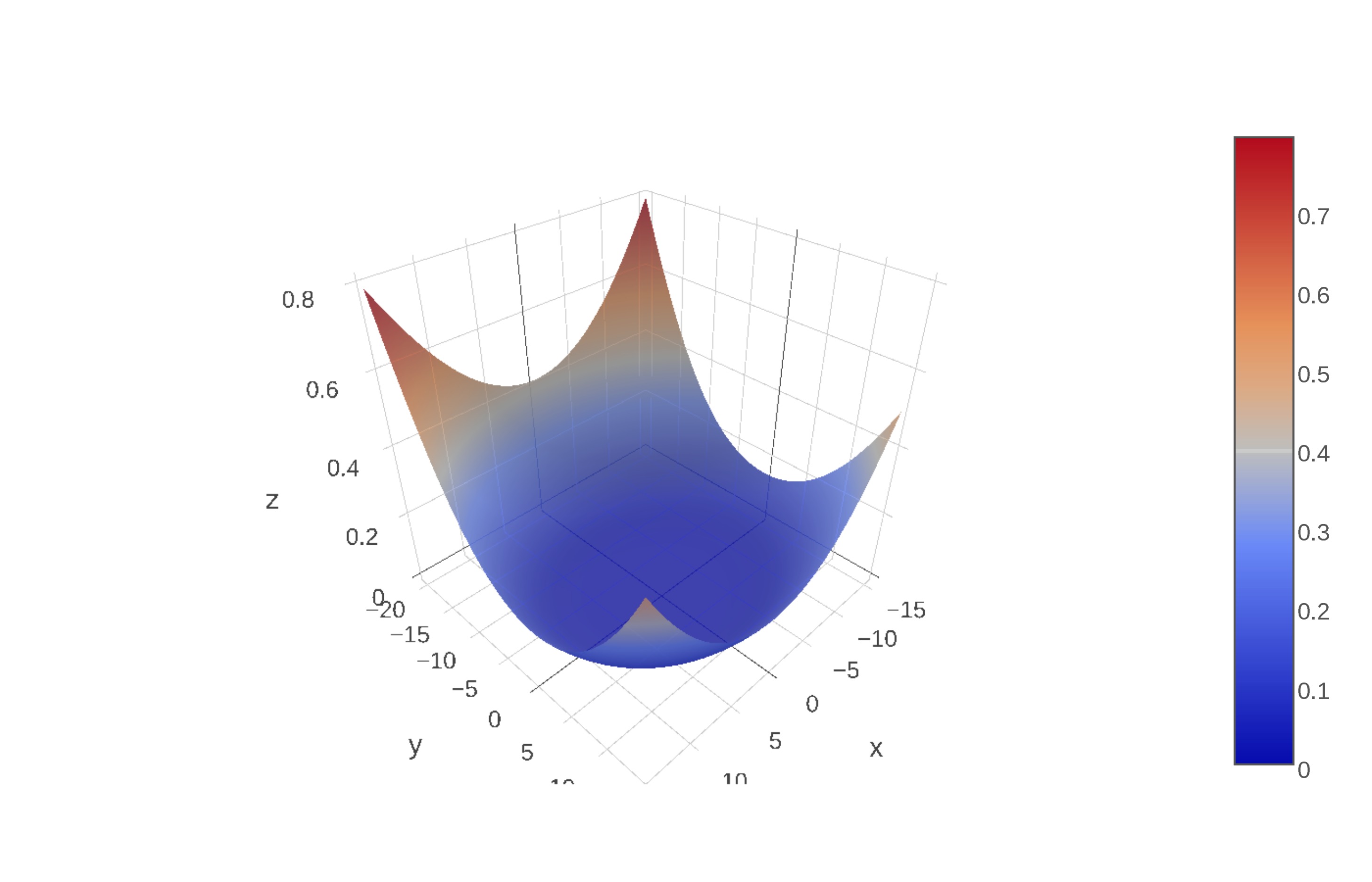}}
\subfloat[]{\includegraphics[width=0.45\textwidth]{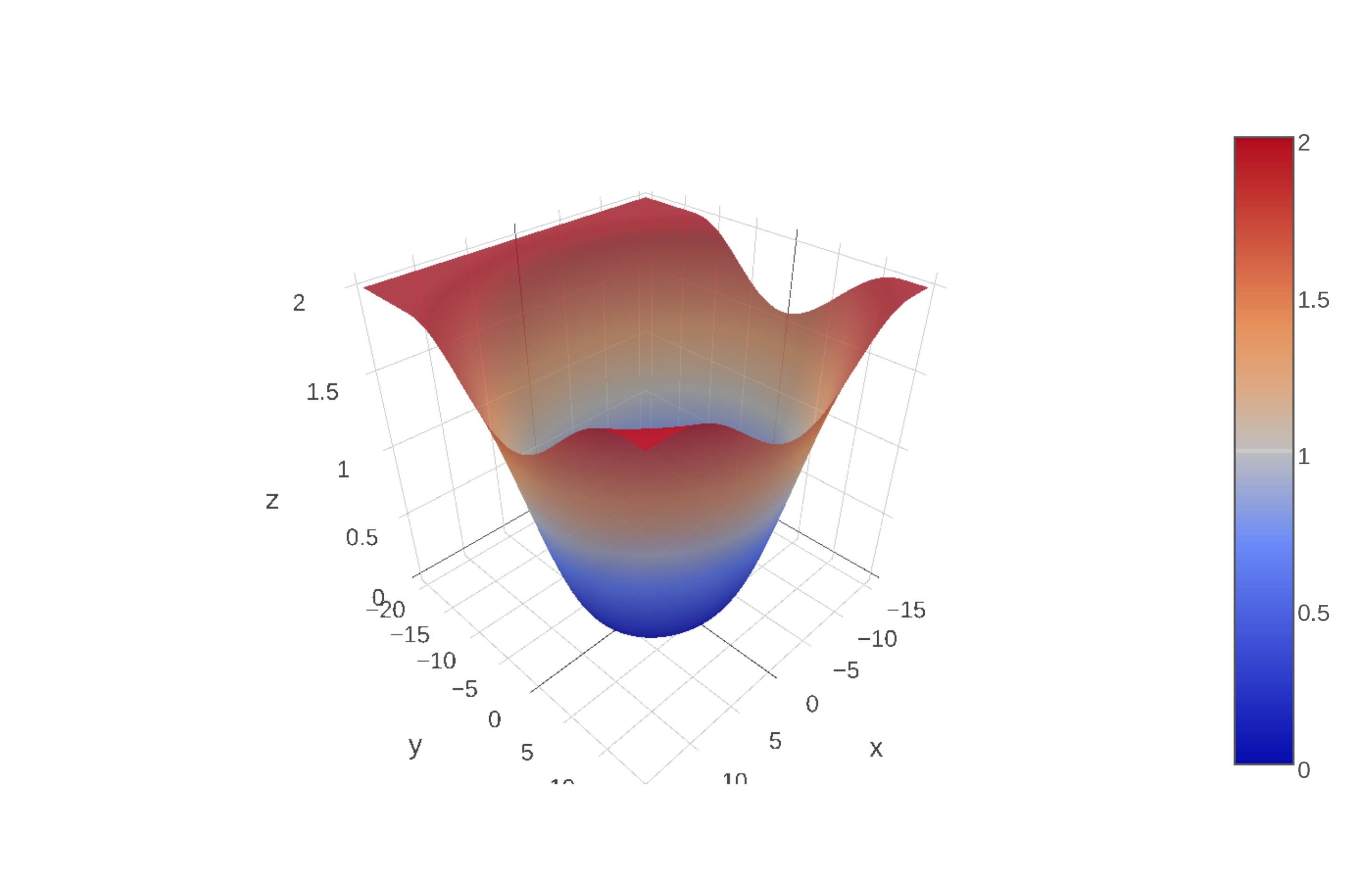}}
\caption[Circular Basin Function Examples]{Examples of the circular basin functions used in constructing the Quadratic-Circle Problem.}
\label{figure-sgd:circular-basin-example}
\end{figure}

The quadratic-circle function is then defined to be the minimum of the quadratic basin function and the circular basin function:
\begin{equation} \label{eqn-sgd:qc-func}
f(x) = \min \lbrace h(x), g(x) \rbrace.
\end{equation}
Examples of the quadratic-circle function are shown in Figure \ref{figure-sgd:qc-example}.

\begin{figure}[htb]
\centering
\subfloat[]{\includegraphics[width=0.45\textwidth]{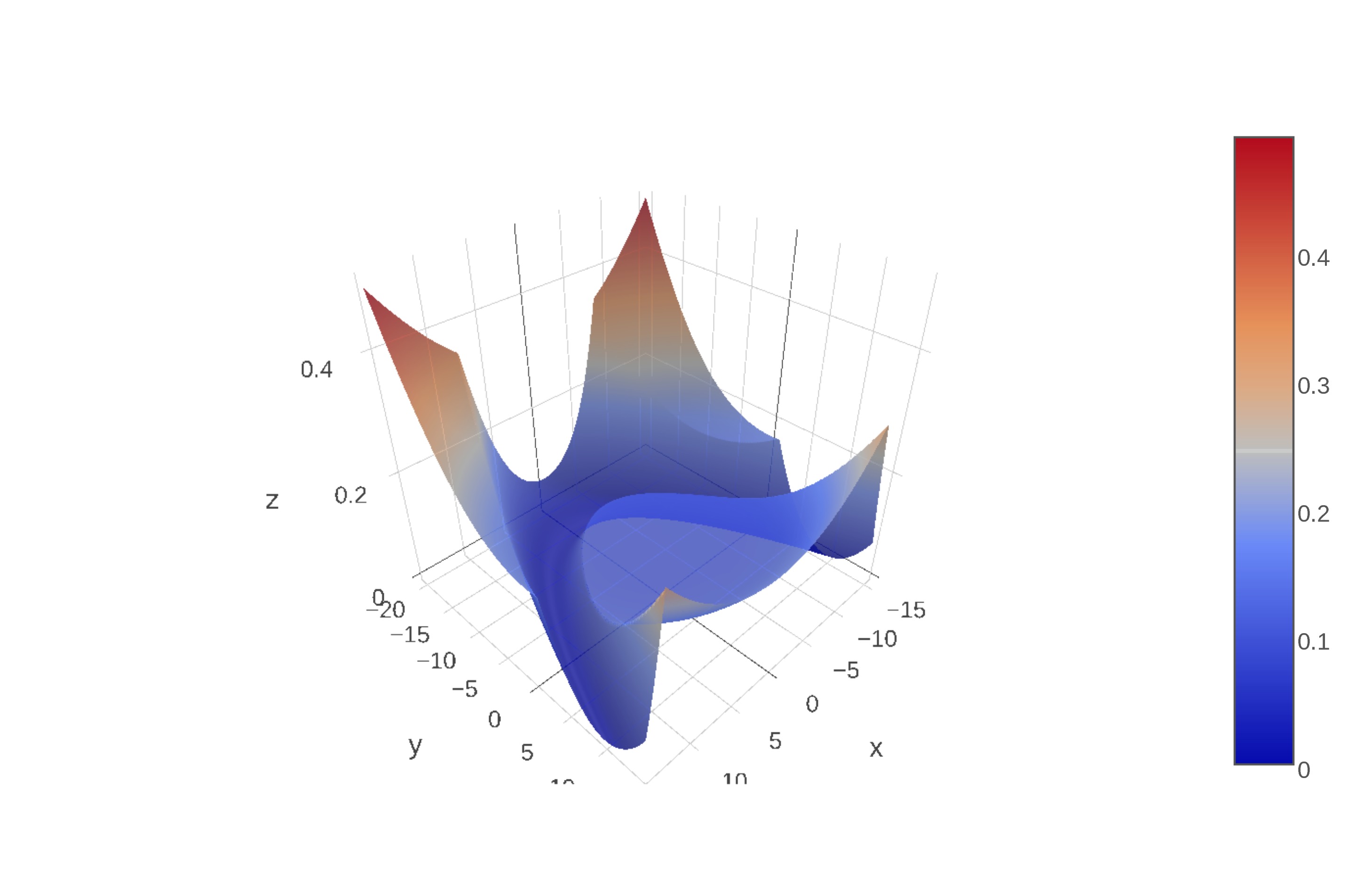}}
\subfloat[]{\includegraphics[width=0.45\textwidth]{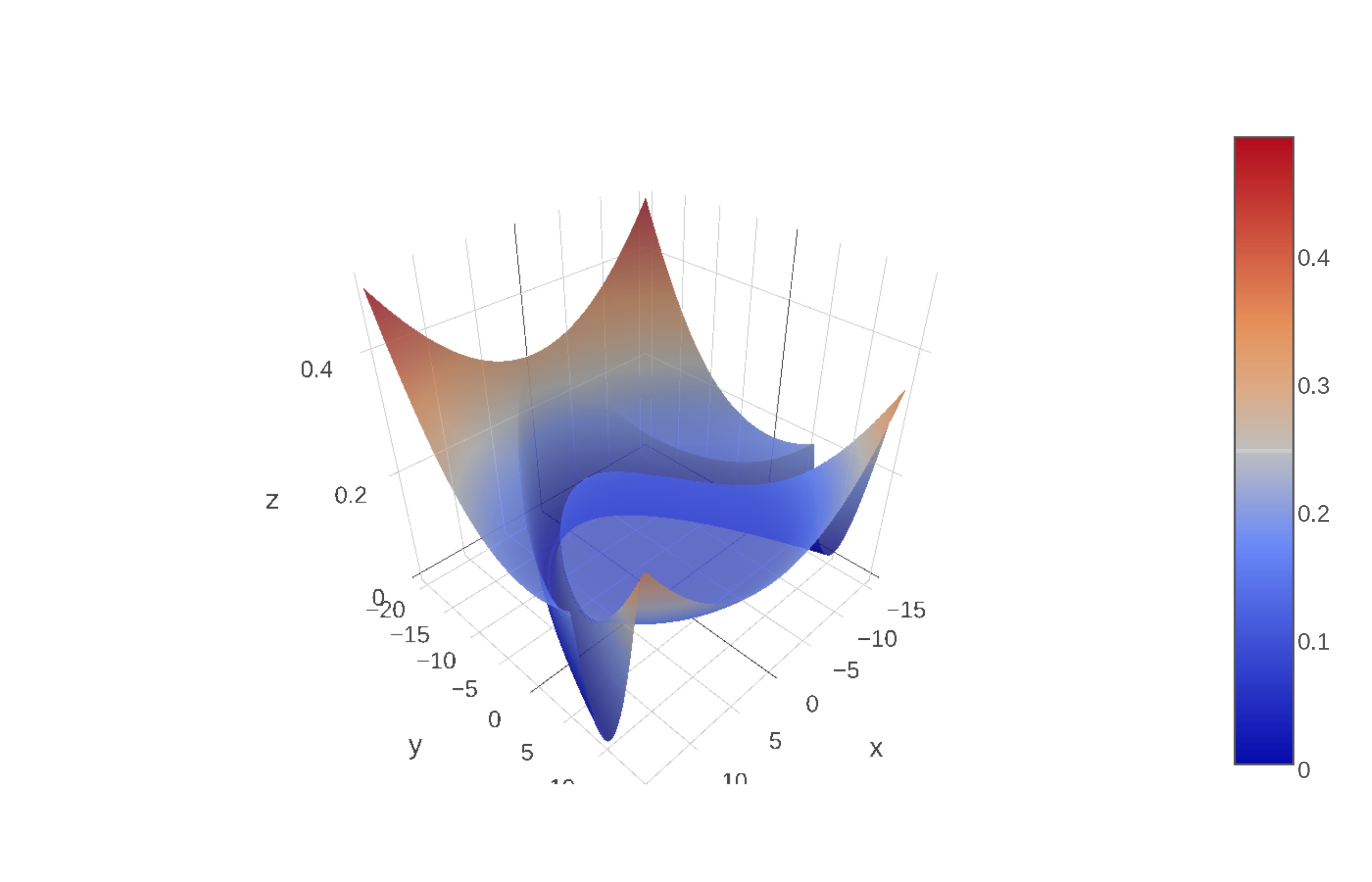}} \\
\subfloat[]{\includegraphics[width=0.45\textwidth]{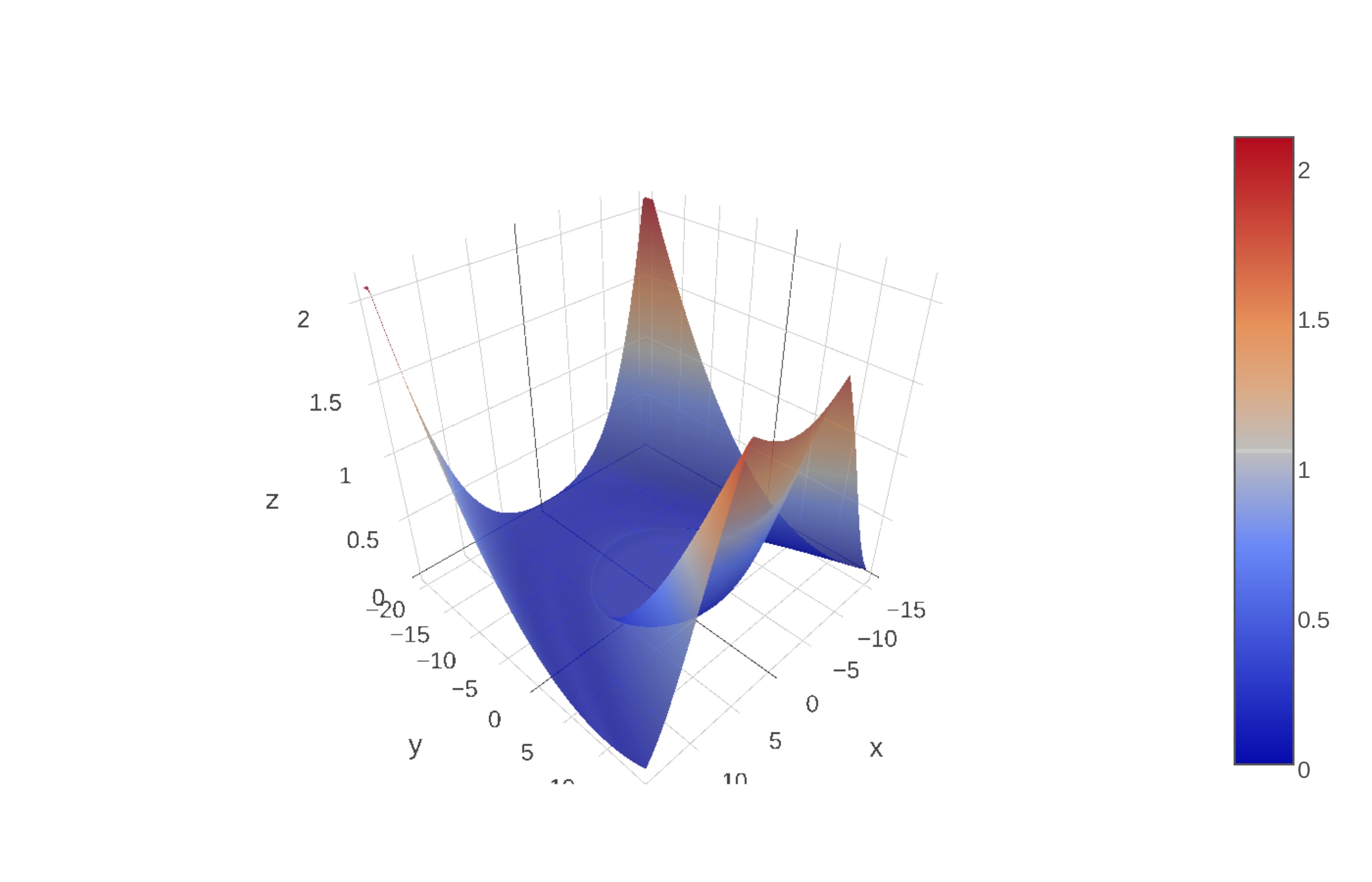}}
\subfloat[]{\includegraphics[width=0.45\textwidth]{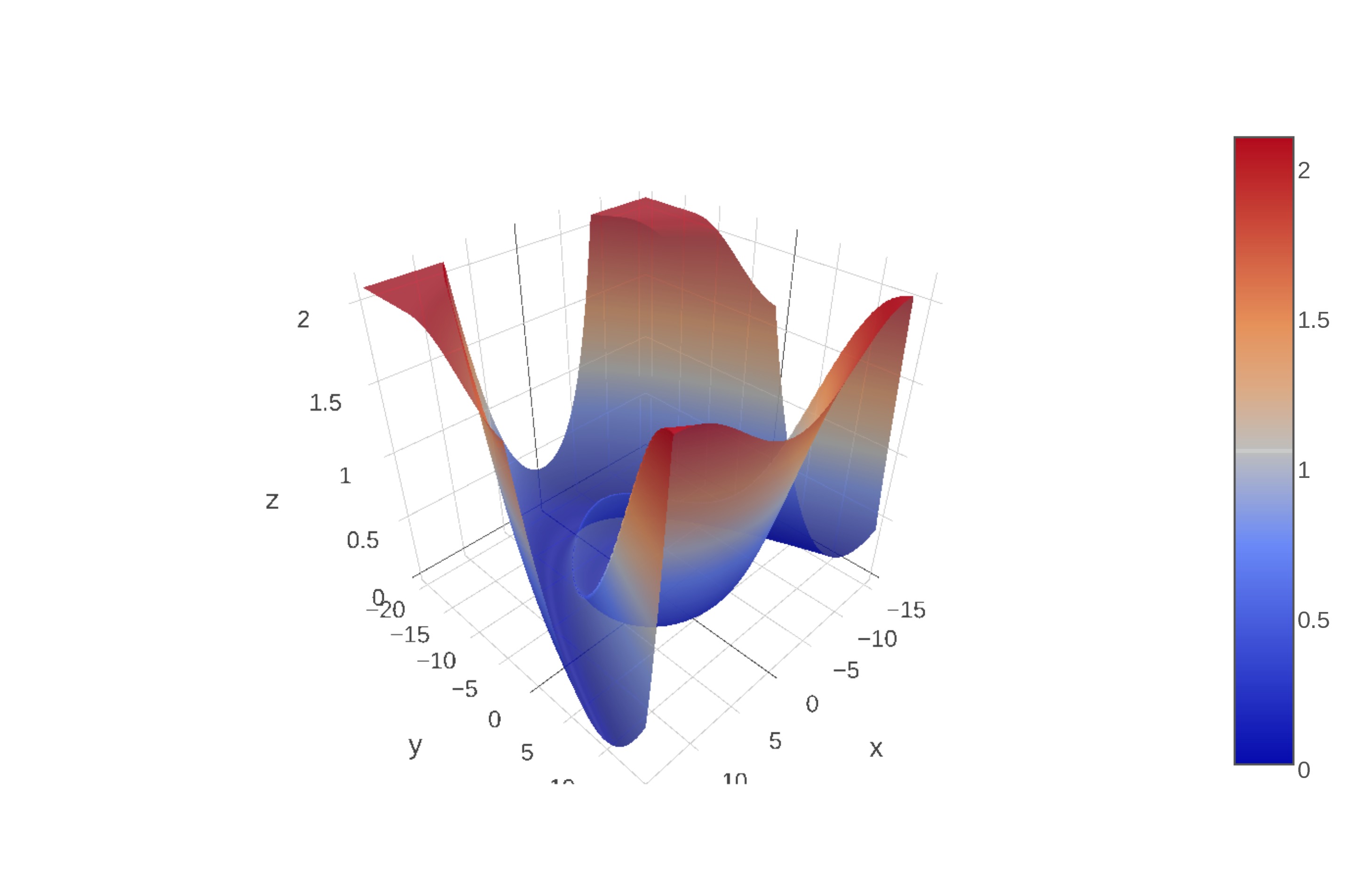}}
\caption[Quadratic-Circle Function Examples]{Examples of the quadratic-circle function.}
\label{figure-sgd:qc-example}
\end{figure}

\begin{remark}
The quadratic-circle function has curves of non-differentiability. If an iterate is at a point of non-differentiability, the gradient will always be selected to be the quadratic component.
\end{remark}

\end{document}